\documentclass[a4paper,11pt,reqno]{amsart}
\usepackage[utf8]{inputenc}
\usepackage[T1]{fontenc}
\usepackage[english]{babel}
\usepackage[
  left=2cm,
  right=2cm,
  top=2.5cm,
  bottom=2.5cm
]{geometry}
\usepackage{mathtools}
\usepackage{amssymb}
\usepackage{mathrsfs}
\IfFileExists{newtxtext.sty}{\usepackage{newtxtext}}{}
\usepackage{microtype}

\usepackage{graphicx}
\usepackage{float}

\usepackage{enumitem}
\usepackage{array,booktabs}

\allowdisplaybreaks 

\numberwithin{equation}{section}

\makeatletter

\long\def\guardarfrag#1#2{%
  \expandafter\long\expandafter\def
  \csname frag@#1\endcsname{#2}%
}

\newcommand{\fragmento}[1]{%
  \@ifundefined{frag@#1}
    {%
      \PackageError{fragmentos}
        {Fragment `#1' is not defined}
        {Check the fragment name.}%
    }
    {%
      \csname frag@#1\endcsname
    }%
}

\makeatother

\newtheorem{proposition}{Proposition}[section]
\newtheorem{theorem}[proposition]{Theorem}
\newtheorem{lemma}[proposition]{Lemma}

\theoremstyle{definition}

\newtheorem{definition}[proposition]{Definition}

\newcommand{\din}{\textnormal{d}}

\newcommand{\loc}{\textnormal{loc}}
\newcommand{\R}{\mathbb{R}}

 \newcommand{\consy}{\mathscr C}

     \newcommand{\xeuc}{\check{x}}
     \newcommand{\yeuc}{\check{y}}
     \newcommand{\zeuc}{\check{z}}
     
     \newcommand{\cil}{\mathcal{M}}
     
\newcommand{\exnum}{q^{\#}}
\newcommand{\exast}{q^\ast}
\newcommand{\centro}{\mathcal{O}}
\newcommand{\exa}{B}

\newcommand{\dcrit}{2^{\ast}}
\newcommand{\dcritb}{2^{\#}}
\newcommand{\spt}{\operatorname{supp}}

\newcommand{\radpa}{\vartheta}

\newcommand{\consob}{\kappa _0}
\newcommand{\consobb}{\kappa _1}
\newcommand{\consobbm}{\kappa _{1,+}}

\newcommand{\consgeo}{\Theta}
\newcommand{\vol}{\operatorname{Vol}}

\newcommand{\lambdaj}{\lambda_j}
\newcommand{\opnj}{\mathcal{N}_g}

\usepackage{etoolbox}

\makeatletter

\patchcmd{\@tocline}
  {\hfil}
  {\leaders\hbox to 0.6em{\hss.\hss}\hfill}
  {}{}

\renewcommand{\@pnumwidth}{1.5em}

\def\l@section{\@tocline{1}{0.6em}{0em}{}{}}         
\def\l@subsection{\@tocline{2}{0.3em}{2em}{}{}}      
\def\l@subsubsection{\@tocline{3}{0.2em}{4.7em}{}{}}   

\patchcmd{\@tocline}
  {\vskip #1}
  {\vskip #1\relax}
  {}{}

\makeatother

\title[Critical Robin problems with noncompact boundary]{Critical Robin Problems on Manifolds with Noncompact Boundary}

\author{Juan Pablo Alcon Apaza}

\address{Juan Pablo Alcon Apaza, Department of Mathematics, Universidade Federal de Minas Gerais, 31270-901, Belo Horizonte, MG, Brazil}

\email{juanpabloalconapaza@gmail.com}


\date{}

\begin{document}
\begin{abstract} 
We work on connected, noncompact $n$-dimensional Riemannian manifolds $(\cil,g)$, $n\ge3$, that are complete with respect to $d_g$ and have noncompact boundary and bounded geometry. We study the existence of {    positive} solutions in $C^\infty(\cil)\cap L^{\dcrit}(\cil)\cap L^\infty(\cil)$ to the Yamabe problem:
\begin{equation*}
\left\{
\begin{aligned}
-\Delta_g u+\frac{n-2}{4(n-1)}R_g u
&=\consy_0u^{\frac{n+2}{n-2}}
&&\text{in }\cil,\\
\frac{\partial u}{\partial\nu}
+\frac{n-2}{2(n-1)}H_g u
&=\consy_1u^{\frac{n}{n-2}}
&&\text{on }\partial{    \cil},
\end{aligned}
\right.
\end{equation*}
where $\consy_0=Y_{\consy}(\cil)$, $\consy_1=Y_{\consy}(\cil)\consy\overline d^{\frac{2}{n-2}}$, the parameter $\consy\ge0$ is fixed, and
\begin{equation*}
Y_{\consy}(\cil)
:=\inf_{u\in C_c^\infty(\cil)\setminus\{0\}}
\frac{E_g(u)}
{\left(\|u\|_{L^{2^\ast}(\cil)}^{2^\ast}
+\consy\|u\|_{L^{2^\#}(\partial\cil)}^{2^\ast}\right)^{2/2^\ast}}.
\end{equation*}
Here $E_g$ is the associated quadratic energy. We assume $Y_{\consy}(\cil)<Y_{\consy,\infty}(\cil)$ and $Y_{\consy,\infty}(\cil)>0$, where $Y_{\consy,\infty}$ is the corresponding variational level at infinity.

We approach this problem using nonnegative solutions $u_j\in W^{1,2}_\partial(B_g(\centro,j))$, $j\in\mathbb N$, {    $j\ge j_0$}, for which we prove that
\[
\sup_{j\ge j_0}\|u_j\|_{L^\infty(B_g(\centro,j-1))}<\infty.
\]
For $\consy>0$, we also prove that, along a subsequence, $\|u_j\|_{L^{\dcritb}(B_g(\centro,j){    \cap\partial\cil})}
\to\overline d$, with $ 0<\overline d\le\consy^{-1/\dcrit}$. For $\consy=0$, we set $\overline d=0$.
\end{abstract}

\maketitle

\tableofcontents

\section{Introduction}\label{6}

The Yamabe problem asks whether a Riemannian metric on a smooth connected compact manifold without boundary can be conformally deformed to a metric of constant scalar curvature. For $n\ge3$, the metric $\widetilde g=u^{4/(n-2)}g$ has constant scalar curvature precisely when the smooth positive function $u$ satisfies
\[
\Delta_g u-\frac{n-2}{4(n-1)}R_g u+Ku^{\frac{n+2}{n-2}}=0,
\qquad K=\frac{n-2}{4(n-1)}R_{\widetilde g}.
\]
Here $\Delta_g$ is the Laplace--Beltrami operator and $R_g$ is the scalar curvature. Following the work of Yamabe \cite{Yamabe1960}, the compact problem was solved by Trudinger \cite{Trudinger1968}, Aubin \cite{Aubin1976}, and Schoen \cite{Schoen1984}. Lee and Parker \cite{LeeParker1987} give a detailed account of its solution.

On complete noncompact manifolds without boundary, the existence of a positive solution and the completeness of the resulting conformal metric are separate questions. Jin \cite{Jin1988} constructs complete metrics whose conformal class contains no complete metric of constant scalar curvature. For the negative scalar curvature problem, Aviles and McOwen \cite{AvilesMcOwen1985,AvilesMcOwen1988} study conformal deformations of complete noncompact manifolds. Li, Tam, and Yang \cite{LiTamYang1998} study related semilinear equations and their geometric applications.

Let $(M^n,g)$ be connected, complete, noncompact, and without boundary. Fix $O\in M$ and write $d(x)=d_g(O,x)$. Following Kim \cite{Kim1996,Kim1997}, consider the Yamabe constant and the Yamabe constant at infinity:
\[
\begin{aligned}
Y(M)&:=\inf_{v\in C_c^\infty(M)\setminus\{0\}}
\frac{\displaystyle\int_M\left(|\nabla_gv|^2+\frac{n-2}{4(n-1)}R_gv^2\right)\,\din v_g}{\|v\|_{L^{\dcrit}(M)}^2},\\
Y_\infty(M)&:=\lim_{r\to\infty}Y\bigl(M\setminus\overline{B_g(O,r)}\bigr).
\end{aligned}
\]
Gro{\ss}e and Nardmann \cite{GrosseNardmann2014} prove that the definition of $Y_\infty(M)$ is independent of the chosen
compact exhaustion and that $Y(M)>-\infty$ whenever
$Y_\infty(M)>0$.

 Zhang \cite{Zhang2003} obtains a positive finite-energy solution with $K=1$ under the assumptions
\[
R_g\ge0,\qquad 0\le Y(M)<Y_\infty(M),\qquad
\operatorname{Vol}_g(B_g(O,r))\le Cr^{n-\rho}
\]
for some $\rho>0$ and all sufficiently large $r$. The solution satisfies
\[
u\in L^{\dcrit}(M),\qquad
\lim_{d(x)\to\infty}d(x)^{(n-2)/2}u(x)=0.
\]

Wei \cite{Wei2019} allows the scalar curvature to change sign. Assuming
\[
Y(M)<Y_\infty(M),\qquad Y_\infty(M)>0,
\]
he proves that the bounds
\[
\begin{gathered}
R_g(x)\ge-Cd(x)^{-2},\qquad
\operatorname{Vol}_g(B_g(O,r))\le Cr^{n+\rho},\\
\rho<\rho_0\bigl(n,Y(M),Y_\infty(M)\bigr),
\end{gathered}
\]
for large $d(x)$ and $r$, where $C>0$ and $\rho_0>0$, give a positive solution satisfying $u(x)=O(d(x)^{-\alpha})$ for some $\alpha>0$. The coefficient $K$ is $1$, $0$, or $-1$ according to the sign of $Y(M)$.

Gro{\ss}e \cite{Grosse2013} uses weighted Sobolev embeddings on connected complete noncompact manifolds of bounded geometry. Her hypotheses include
\[
Y(M)<Y_\infty(M),\qquad
\inf_{v\in C_c^\infty(M)\setminus\{0\}}
\frac{\displaystyle\int_M\left(|\nabla_gv|^2+\frac{n-2}{4(n-1)}R_gv^2\right)\,\din v_g}
{\|v\|_{L^2(M)}^2}>0,
\]
and her solution is a positive bounded minimizer of the Yamabe quotient. Wei, in \cite[Theorem~1.5]{Wei2019}, instead obtains a bounded
positive solution under the assumptions
\[
Y(M)<Y_\infty(M),\qquad Y_\infty(M)>0,\qquad
|\operatorname{Ric}_g|\le C_0,\qquad
\operatorname{inj}(M,g)\ge i_0>0.
\]
These existence results do not assert that the conformal metric
is complete.

\subsection{The Yamabe problem on manifolds with boundary}

On a smooth Riemannian manifold with boundary, the conformal metric $\widetilde g=u^{4/(n-2)}g$ has constant scalar curvature and constant boundary mean curvature precisely when $u>0$ satisfies
\begin{equation}\label{1}
\left\{
\begin{aligned}
-\Delta_g u+\frac{n-2}{4(n-1)}R_g u
&=\mathscr C_0u^{\frac{n+2}{n-2}}
&&\text{in }\operatorname{int}M,\\
\frac{\partial u}{\partial\nu}
+\frac{n-2}{2(n-1)}H_g u
&=\mathscr C_1u^{\frac n{n-2}}
&&\text{on }\partial M.
\end{aligned}
\right.
\end{equation}
Here we use the conventions for $H_g$ and $\nu$ stated after \eqref{e1}. In these conventions,
\[
R_{\widetilde g}=\frac{4(n-1)}{n-2}\mathscr C_0,
\qquad H_{\widetilde g}=\frac{2(n-1)}{n-2}\mathscr C_1.
\]

First suppose that $M$ is compact and connected, with smooth nonempty boundary. Cherrier \cite{zbMATH03879780} studies nonlinear Neumann problems on compact Riemannian manifolds and applies his existence and regularity results to prescribed scalar curvature and boundary mean curvature. Escobar \cite{Escobar1992,EscobarScalarFlat1992} studies the two problems
\[
\left\{
\begin{aligned}
R_{\widetilde g}&=\text{constant}&&\text{in }M,\\
H_{\widetilde g}&=0&&\text{on }\partial M,
\end{aligned}
\right.
\qquad\text{and}\qquad
\left\{
\begin{aligned}
R_{\widetilde g}&=0&&\text{in }M,\\
H_{\widetilde g}&=\text{constant}&&\text{on }\partial M.
\end{aligned}
\right.
\]
With $E_g$ defined by \eqref{e1},   the associated constants are
\[
\begin{aligned}
Y(M,\partial M)
&:=\inf_{v\in C^\infty(M)\setminus\{0\}}
\frac{E_g(v)}{\|v\|_{L^{\dcrit}(M)}^2},\\
Q(M,\partial M)
&:=\inf_{\substack{v\in C^\infty(M)\\v|_{\partial M}\not\equiv0}}
\frac{E_g(v)}{\|v\|_{L^{\dcritb}(\partial M)}^2}.
\end{aligned}
\]
Escobar \cite{Escobar1992} proves that
\[
Y(M,\partial M)<Y(S^n_+,\partial S^n_+)
\]
gives a smooth positive minimizer for the problem with minimal boundary. For the scalar-flat problem, Escobar \cite{EscobarScalarFlat1992} uses the strict comparison with the Euclidean ball.  The sufficient condition is
\[
-\infty<Q(M,\partial M)<Q(B^n,\partial B^n).
\]
The assumption $Q(M,\partial M)>-\infty$ is also stated by
Marques \cite{Marques2005} and
Almaraz \cite{Almaraz2010}.
Here $S^n_+$ has the round metric and $B^n$ has the Euclidean metric.

Marques \cite{Marques2005} and Almaraz \cite{Almaraz2010} establish the strict inequality for $Q(M,\partial M)$ on compact manifolds with umbilic boundary under dimensional and Weyl curvature hypotheses. Brendle and Chen \cite{BrendleChen2014} study the problem with minimal boundary in the umbilic case. For the scalar-flat problem, Mayer and Ndiaye \cite{MayerNdiaye2017} use the barycenter technique to prove existence when
\[
n\ge6,\qquad \partial M\text{ is umbilic},\qquad Q(M,\partial M)>0.
\]
Their conclusion is the existence of a solution; it does not assert that the solution minimizes $Q(M,\partial M)$.

For compact manifolds with $Y(M,\partial M)>0$, Han and Li
\cite{HanLi1999} ask whether \eqref{1} admits
a positive solution for every
\[
\mathscr C_0>0,\qquad \mathscr C_1\in\mathbb R.
\]
They prove existence for locally conformally flat manifolds
with umbilic boundary. The same authors, in
\cite{HanLi2000}, prove existence for $n\ge5$
when the boundary has a nonumbilic point. Chen, Ruan, and Sun \cite{ChenRuanSun2019}
establish the Han--Li conjecture for $3\le n\le7$ and prove
it in additional cases for $n\ge8$.

On compact manifolds with nonempty boundary, Chen and Sun
\cite{ChenSun2019} define
\[
Y_{a,b}(M,\partial M)
:=\inf_{v\in H^1(M)\setminus\{0\}}
\frac{\displaystyle\frac{4(n-1)}{n-2}E_g(v)}
{a\|v\|_{L^{\dcrit}(M)}^2
+2(n-1)b\|v\|_{L^{\dcritb}(\partial M)}^2},
\qquad a,b>0,
\]
and prove that the infimum is attained by a smooth positive
function if
\[
Y_{a,b}(M,\partial M)<Y_{a,b}(S^n_+,\partial S^n_+).
\]
Moreover,  they verify
this inequality when $Y(M,\partial M)>0$, $M$ is not conformally
equivalent to the hemisphere, and either $3\le n\le7$ or $M$
is spin. For fixed $a,b$, the two curvature constants are
determined by the minimizer and a common Lagrange multiplier.
To prescribe the constants independently, Han and Li
\cite[Section~1]{HanLi2000} and Chen, Ruan, and Sun
\cite[Section~1]{ChenRuanSun2019} use a mountain pass argument.

In the noncompact case, Maggi and Neumayer
\cite{MaggiNeumayer2017} study sharp Sobolev and trace
inequalities on the Euclidean half-space $\mathbb R^n_+$.
In particular, they consider
\[
\Phi(T):=\inf\left\{
\|\nabla v\|_{L^2(\mathbb R^n_+)}\mid
\|v\|_{L^{\dcrit}(\mathbb R^n_+)}=1,\ 
\|v(\cdot,0)\|_{L^{\dcritb}(\mathbb R^{n-1})}=T
\right\},\qquad T>0,
\]
over functions $v\in L^{\dcrit}(\mathbb R^n_+)$ with
$\nabla v\in L^2(\mathbb R^n_+)$.
The authors, in \cite[Theorem~1.1]{MaggiNeumayer2017},
describe the nonnegative minimizers up to dilations and
translations parallel to the boundary.

For a complete noncompact manifold with compact nonempty boundary and finitely many ends, Schwartz \cite{Schwartz2006} assumes
\[
R_g=0,\qquad H_g>0,\qquad
\operatorname{Ric}_g\ge-\frac{(n-1)A}{(1+d_g(O,x))^2}g,\qquad A\ge0,
\]
and
\[
\int_1^\infty\frac{t}{\operatorname{Vol}_g(E\cap B_g(O,t))}\,\din t<\infty
\quad\text{for every end }E.
\]
For every $f\in C^\infty(\partial M)$, he constructs a conformal metric satisfying
\[
R_{\widetilde g}=0,\qquad H_{\widetilde g}=f,\qquad
\widetilde g=u^{4/(n-2)}g,\qquad 0<\epsilon\le u\le C.
\]
The positive lower bound ensures that $\widetilde g$ is complete.

For complete noncompact manifolds with possibly noncompact boundary, Albanese and Rigoli \cite{AlbaneseRigoli2017} prove weak maximum principles under volume growth and boundary hypotheses and apply them to conformal diffeomorphisms. Cruz and Santos \cite{CruzSantos2025} consider complete noncompact manifolds with noncompact boundary and prove that
\[
R_g\le-c_0<0,\qquad H_g=0
\]
allow every bounded smooth function $f<0$ to be realized by a complete conformal metric with
\[
R_{\widetilde g}=f\quad\text{in }M,\qquad
H_{\widetilde g}=0\quad\text{on }\partial M.
\]

Alcon and Almaraz \cite{AlconApazaAlmaraz2024} start from a compact space $M$ locally modeled on solid cones and a nonempty finite union $\Gamma=\bigcup_m\Gamma_m$ of compact submanifolds satisfying their geometric assumptions. The smooth manifold $M\setminus\Gamma$ is noncompact; its boundary is noncompact when $\Gamma\cap\partial M\ne\varnothing$. For any constants $r_0,h_0>0$, they obtain a conformal metric on $M\setminus\Gamma$, complete as a metric space, with
\[
R_{\widetilde g}=-r_0,\qquad H_{\widetilde g}=-h_0,
\]
if and only if
\[
\dim\Gamma_m>\frac{n-2}{2}\qquad\text{for every }m.
\]
The conformal factor satisfies
\[
\liminf_{x\to\Gamma}u(x)\operatorname{dist}_g(x,\Gamma)^{(n-2)/2}>0.
\]

\subsection{Main result}

We work on a connected, noncompact $n$-dimensional Riemannian manifold $(\cil,g)$, $n\ge3$, with noncompact boundary. We assume that $(\cil,d_g)$ is complete and that $\cil$ has bounded geometry as in Definition~\ref{bg}.

For $u\in C_c^\infty(\cil)$, define the quadratic form associated with
\[
\mathcal L_g=-\Delta_g+\frac{n-2}{4(n-1)}R_g,
\qquad
\mathcal B_g=\frac{\partial}{\partial\nu}
+\frac{n-2}{2(n-1)}H_g
\]
by
\begin{equation}\label{e1}
E_g(u):=\int_{\cil}
\left(|\nabla_gu|^2+\frac{n-2}{4(n-1)}R_gu^2\right)\,\din v_g
+\int_{\partial\cil}\frac{n-2}{2(n-1)}H_gu^2\,\din s_g.
\end{equation}
The same expression is used for $u\in W^{1,2}(\cil)$ whenever its terms are finite. Here $\Delta_g=\operatorname{div}_g\nabla_g$, $\nu$ is the outward unit normal along $\partial\cil$, and
\[
\mathrm{II}(X,Y):=g(\nabla_X\nu,Y),\qquad
H_g:=\operatorname{tr}_{g|_{\partial\cil}}\mathrm{II},
\qquad X,Y\in T(\partial\cil).
\]
For these conventions, the conformal change $\widetilde g=u^{4/(n-2)}g$ satisfies
\[
\mathcal L_gu=\frac{n-2}{4(n-1)}R_{\widetilde g}u^{\frac{n+2}{n-2}},
\qquad
\mathcal B_gu=\frac{n-2}{2(n-1)}H_{\widetilde g}u^{\frac n{n-2}}.
\]

For a compact set $K\subset\cil$, define
\[
C_c^\infty(\cil\setminus K)
:=\{u\in C_c^\infty(\cil):\operatorname{supp}u\subset\cil\setminus K\}.
\]
For a fixed $\consy\ge0$, define
\begin{equation}\label{y2}
Y_{\consy}(\cil)
:=\inf_{u\in C_c^\infty(\cil)\setminus\{0\}}
\frac{E_g(u)}
{\left(\|u\|_{L^{{\dcrit}}(\cil)}^{{\dcrit}}
+\consy\|u\|_{L^{{\dcritb}}(\partial\cil)}^{{\dcrit}}\right)^{2/{\dcrit}}},
\end{equation}
and define $Y_{\consy}(\cil\setminus K)$ by restricting the infimum to nonzero functions in $C_c^\infty(\cil\setminus K)$. For a compact exhaustion satisfying
\[
K_i\subset\operatorname{int}_{\cil}K_{i+1},\qquad
\bigcup_{i\ge1}K_i=\cil,
\]
set
\begin{equation}\label{y3}
Y_{\consy,\infty}(\cil)
:=\lim_{i\to\infty}Y_{\consy}(\cil\setminus K_i).
\end{equation}
The limit is independent of the exhaustion.

Define also
\begin{equation}\label{m1}
\Lambda^+_{n,\consy}
:=\inf_{v\in C_c^\infty(\overline{\mathbb R^n_+})\setminus\{0\}}
\frac{\displaystyle\int_{\mathbb R^n_+}|\nabla v|^2\,\din x}
{\left(\|v\|_{L^{{\dcrit}}(\mathbb R^n_+)}^{{\dcrit}}
+\consy\|v(\cdot,0)\|_{L^{{\dcritb}}(\mathbb R^{n-1})}^{{\dcrit}}\right)^{2/{\dcrit}}}.
\end{equation}
We assume
\begin{equation}\label{r2}
\begin{gathered}
Y_{\consy}(\cil)<Y_{\consy,\infty}(\cil),
\qquad Y_{\consy,\infty}(\cil)>0,\\
Y_{\consy}(\cil)<\Lambda^+_{n,\consy}.
\end{gathered}
\end{equation}
By Lemma~\ref{l2},
\[
-\infty<Y_{\consy}(\cil)
\le Y_{\consy,\infty}(\cil)
\le\Lambda^+_{n,\consy}.
\]
Consequently, the last inequality in \eqref{r2} follows from the first.

Fix $\centro\in\cil$ and set $\exa_j:=B_g(\centro,j)$. On these domains we use
\[
W^{1,2}_{\partial}(\exa_j)
=\overline{C_c^\infty(\exa_j)}^{\,W^{1,2}(\exa_j)},
\]
as in \eqref{w1}.

Define
\[
\lambdaj:=\inf_{v\in W^{1,2}_{\partial}(\exa_j)\setminus\{0\}}
\frac{E_g(v)}
{\left(\|v\|_{L^{{\dcrit}}(\exa_j)}^{{\dcrit}}
+\consy\|v\|_{L^{{\dcritb}}(\exa_j\cap\partial\cil)}^{{\dcrit}}\right)^{2/{\dcrit}}},
\]
where $v$ is identified with its zero extension to $\cil$. Lemma~\ref{l5} gives $\lambda_j\searrow Y_{\consy}(\cil)$. For all sufficiently large $j$, Proposition~\ref{r29} supplies a nonnegative minimizer $u_j$ with
\[
\|u_j\|_{L^{{\dcrit}}(\exa_j)}^{{\dcrit}}
+\consy\|u_j\|_{L^{{\dcritb}}(\exa_j\cap\partial\cil)}^{{\dcrit}}=1,
\qquad E_g(u_j)=\lambda_j.
\]

Set $d_j:=\|u_j\|_{L^{{\dcritb}}(\exa_j\cap\partial\cil)}$. The Euler--Lagrange equations for $u_j$ are
\[
\left\{
\begin{aligned}
\mathcal L_gu_j&=\lambda_j u_j^{{\dcrit}-1}
&&\text{in }\exa_j\cap\operatorname{int}\cil,\\
\mathcal B_gu_j&=\lambda_j\consy d_j^{{\dcrit}-{\dcritb}}u_j^{{\dcritb}-1}
&&\text{on }\exa_j\cap\partial\cil.
\end{aligned}
\right.
\]

Following the exhaustion argument of Wei \cite{Wei2019},
we construct solutions on the sets $\exa_j$. The main estimate, proved by Moser iteration in Proposition~\ref{r12}, is
\[
\sup_{j\ge j_0}\|u_j\|_{L^\infty(\exa_{j-1})}<\infty.
\]
Together with local elliptic regularity and the strict inequality $Y_{\consy}(\cil)<Y_{\consy,\infty}(\cil)$, this estimate gives a nonzero limit along a subsequence and yields the following result.
\begin{theorem}\label{r25}
Under the assumptions of Definition~\ref{bg} and \eqref{r2}, with $\cil$ connected, $\partial\cil$ noncompact, and $\consy\ge0$, there exist a function $u$ and a constant $\overline d\ge0$ such that
\[
u\in C^\infty(\cil)\cap L^{\dcrit}(\cil)\cap L^\infty(\cil),
\qquad u>0\text{ on }\cil,
\]
and
\begin{equation}\label{r26}
\left\{\begin{aligned}
\mathcal L_gu&=Y_{\consy}(\cil)u^{\dcrit-1}
&&\text{in }\operatorname{int}\cil,\\
\mathcal B_gu&=Y_{\consy}(\cil)\consy\overline d^{\dcrit-\dcritb}u^{\dcritb-1}
&&\text{on }\partial\cil.
\end{aligned}\right.
\end{equation}
If $\consy>0$, then, along a subsequence,
\[
\|u_j\|_{L^{\dcritb}(\exa_j\cap\partial\cil)}\longrightarrow\overline d,\qquad 0<\overline d\le\consy^{-1/\dcrit}.
\]
For $\consy=0$, we set $\overline d=0$.
\end{theorem}


%
%
%

%
%
%
%
%

\section{Preliminaries}\label{2}
We use the following notation.
\begin{enumerate}[label=$(\roman*)$]
\item For $a\in \R$, define $a^+:= \max \{a,0\}$ and $a^-:=\max\{-a,0\}$.

\item Points of $\cil$ are denoted by $x,y,z$, and points of $\R^n$ by $\xeuc$, $\yeuc$ and $\zeuc$. Points of $\R^{n-1}$ are denoted by $\xeuc^\prime$, $\yeuc^\prime$ and $\zeuc^\prime$.

\item $B_r^n := \{\xeuc \in \R^n \mid |\xeuc|<r\}$, $B_r^{n-1} := \{\xeuc^\prime \in \R^{n-1} \mid |{\xeuc^\prime}|<r\}$
$$
Q_{i, r}:= \left\{\begin{aligned}&B_r^{n-1} \times[0, r) && \text { if } i=1 \\ 
&B_r^n && \text { if } i=0\end{aligned}\right.
$$

\item $\mathbb R^n_+:=\{(\xeuc ^\prime,t)\in\mathbb R^{n-1} \times\mathbb R \mid t>0\}$

\item For $A\in\mathbb R^{n\times n}$, the operator norm used below is
\begin{equation*}
 \|A\|_{\mathrm{op}}
 :=\sup_{\xi\neq0}\frac{|A\xi|}{|\xi|}
\end{equation*}
The vector norm is Euclidean.

\item Matrix inequalities $aI\leq G\leq bI$ denote
\[
 a|\xi|^2\leq\xi^TG\xi\leq b|\xi|^2
 \quad\forall\xi\in\mathbb R^n.
\]

\item $q^\ast := \frac{nq}{n-q}$ and $q^\#=\frac{q(n-1)}{n-q}$.

\item We write $C(a_1,a_2, \ldots,a_m)$ for a constant $C$ that depends on $a_i$, $i=1,\ldots, m$. When any of these quantities depends on the constants $\consgeo _C, \consgeo _ \partial, \consgeo _0$ {and the curvature and second-fundamental form bounds $c_k$}   in Definition~\ref{bg}, we use the abbreviated notation $C(\ldots , g)$.

\item The gradient on $(\cil , g) $ is denoted by $\nabla _g$, whereas the gradient on $\R^n$ is denoted simply by $\nabla$.
\end{enumerate}

\subsection{Uniform coordinate charts and partitions of unity}
Throughout this subsection, we assume that $n\ge 2$ and $q>1$.

\begin{definition}[{\cite[Definition 2.2]{zbMATH01606240}}] \label{bg}
Let $({\cil}^n,g)$ be a smooth Riemannian manifold with boundary 
$\partial {\cil}\neq\varnothing$. Every connected component, including its boundary, is complete with respect to
its intrinsic length distance.

The manifold has {\bf bounded geometry} if the following conditions hold:
\begin{enumerate}[label=$(\roman*)$]
 \item \label{c}  There exists $\consgeo _C>0$
such that $F:\partial {\cil}\times[0,{\consgeo _C})\to {\cil}$, defined by $ F(z,t):=\exp_z(-t\nu(z))$,  is a diffeomorphism onto its image;

\item \label{b} The injectivity radius $\consgeo _{\partial} := r_{\operatorname{inj}}(\partial {\cil})$ of $\partial {\cil}$  is positive;
 
 \item \label{i}  There exists ${ 0<{\consgeo _0}<{\consgeo _C}}$ such that $\exp_x :B(0,r)\subset T_x{\cil}\to {\cil}\setminus\partial {\cil}$ is a diffeomorphism onto its image for $0<r\leq {\consgeo _0}$, $x\in {\cil}\setminus N(r)$, where $N(r)=F(\partial {\cil}\times[0,r])$;
 
 \item \label{cv} $\|\nabla^k\operatorname{Rm}_g\|_{L^\infty({\cil})}
  +\|(\nabla^h)^k\mathrm{II}\|_{L^\infty(\partial {\cil})}
 \leq c_k<\infty$, $k=0,1,2,\ldots$.
\end{enumerate}
{    Here $h$ is the metric induced by $g$ on $\partial\cil$, and $\nabla^h$ is its Levi--Civita connection.}

\end{definition}

For all  $r>0$ and $i=0,1$, define
\begin{equation}\label{17}
Q_{i,r}:=\left\{\begin{aligned}
&B^{n-1}_r \times [0,r) & & \text { if } i=1\\
& B^n_r & & \text { if } i=0
\end{aligned}
\right.
\end{equation}
\begin{lemma} \label{chs}
There exist $r_1\in (0, 2^{-1}\min \{{\consgeo _0}, {\consgeo _C}\})$ (see Definition~\ref{bg}) and constants $C_1,C_0>0$, independent of $z\in {\cil}$, and charts 
\begin{equation}\label{ch}
\left\{\begin{aligned} \Phi_{1,z}  : (\xeuc',t)\in Q_{1, r_1}&\mapsto F\bigl(\exp_z^{\partial {\cil}}(E _{1,z} \xeuc'),t\bigr) & & \text { if } z\in \partial \cil,\\
\Phi_{0,z}  : \xeuc\in Q_{0, r} & \mapsto \exp_z \left(E _{0,z} \xeuc\right), \ r= 2^{-1} \min \{  r_1, d_g (z, \partial \cil) \}  & & \text { if }  z\in \cil  \setminus  \partial \cil
\end{aligned}\right.
\end{equation}
where $E_{1,z}:\mathbb R^{n-1}\to T_z\partial {\cil}$ and $E_{0,z}:\mathbb R^{n}\to T_z {\cil}$  are linear isometries. In the coordinates given by $\Phi _{i,z}$, $i=0,1$, we have 
\begin{gather}
 [g_{ij}](0,0)=I,\qquad
 \sup_z\sup_{{\operatorname{dom}\Phi_{h,z}}}|\partial_k g_{ij}|,\quad \sup_z\sup_{{\operatorname{dom}\Phi_{h,z}}}|\partial_k g^{ij}|\leq C_1,\ h=0,1;
 \label{dv}\\
 \|{[g_{ij}]^{-1}}(\xeuc ',t)-I\|_{\mathrm{op}}, \quad \|[g_{ij}](\xeuc ',t)-I\|_{\mathrm{op}}
 \leq C_0(|\xeuc '|+t),\qquad (\xeuc ^\prime ,t)\in Q_{1,r_1}, \label{cl}\\
 \|{[g_{ij}]^{-1}}(\xeuc)-I\|_{\mathrm{op}}, \quad \|[g_{ij}](\xeuc )-I\|_{\mathrm{op}}
 \leq C_0|\xeuc|,\qquad \xeuc \in Q_{0, \frac{ \min \{  r_1, d_g (z, \partial \cil)}{2}},\label{12}
\end{gather}
where $C_0=n{    \sqrt n}\,C_1$.

For ${\radpa}>1$, set
\begin{equation}\label{rl}
 r_{\radpa}:=
 \min\left\{r_1,
       \frac{1-{\radpa}^{-1}}{4(1+C_0)}\right\}>0.
\end{equation}
If $x\in \cil \setminus \partial \cil$, we define
\begin{equation}\label{13}
 r_{x,{\radpa}}:=
 \min\left\{\frac{r_1}{2} , \frac{d_g (x, \partial \cil)}{2},  \frac{1-{\radpa}^{-1}}{4(1+C_0)}\right\}>0.
\end{equation}

On $Q_{1,r_{\radpa}}$ or $Q_{0,r_{x,{\radpa}} }$, we have
\begin{equation}\label{mt}
 {\radpa}^{-1}I\leq [g_{ij}] \leq{\radpa} I,
 \qquad {\radpa}^{-1}I\leq [g_{ij}]^{-1}\leq{\radpa} I.
\end{equation}
\end{lemma}

\begin{proof}

\textbf{Step 1.}
Let $z\in  \cil$. Set $c =\max\{1,c_0,c_1\}$, where the constants are given in Definition~\ref{bg}\,\ref{cv}. By
\cite[Theorem 2.5 (a1), (b1), (c)]{zbMATH01606240},
there exist uniform constants $R_1,R_2, R_3>0$ and $C_1>0$ for which $\Phi _{1,z} : B^{n-1}_{r_1} \times [0,r_2) \to \cil$, with $(r_1, r_2)\in (0,R_1]\times (0,R_2]$, is a normal boundary chart, and $\Phi _{0,z} : B^{n}_{r_3} \to \cil\setminus \partial \cil$ is a Gaussian chart for $r_3\in (0,\min\{R_3 , {\consgeo _0} \}]$ and $z\in \cil \setminus F(\partial \cil \times [0,r_3])$. Furthermore, in the coordinates $\Phi _{i,z}$, $i=0,1$, we have
\[
|g_{ij}|, \ |g^{ij}|, \ |\partial_k g_{ij}|, \ |\partial_k g^{ij}| \leq C_1,
 \qquad 1\leq i,j,k\leq n,
\]

Using Definition~\ref{bg}, choose
\[
 r_1=\frac12\min\left\{{\consgeo _C}, \frac{\consgeo _0}{2}, {\consgeo _{\partial}} , R_1 , R_2,R_3\right\}.
\]
This proves the derivative bound in \eqref{dv}.

\medskip

\noindent \textbf{Step 2.} We first prove \eqref{cl} for $z\in\partial\cil$. For $(\xeuc ^\prime,t)\in Q_{{1},r_1}$ and $0\leq s\leq1$,
$(s\xeuc ^\prime,st)\in Q_{{1},r_1}$. Therefore
\begin{align*}
 g_{ij}(\xeuc ^\prime,t)-\delta_{ij}
 &=\int_0^1\frac{\din }{\din s}g_{ij}(s\xeuc ^\prime,st)\,\din  s\\
 &=\int_0^1\left[
      \sum_{\alpha=1}^{n-1} x^\alpha
           \partial_\alpha g_{ij} (s\xeuc ^\prime,st)
       +t\,\partial_n g_{ij} (s\xeuc ^\prime,st)\right]\, \din s.
\end{align*}

Hence
\begin{equation}\label{en}
 |g_{ij}(\xeuc ^\prime,t)-\delta_{ij}|
 \leq C_1\left(\sum_{\alpha=1}^{n-1}|x^\alpha|+t\right)
 \leq C_1\bigl(\sqrt  {n-1} \,|\xeuc ^\prime|+t\bigr).
\end{equation}

For any matrix $A=[a_{ij}]$ and $|\xi|=1$,
\[
 |(A\xi)_i|
 \leq\max_{k,l}|a_{kl}|\sum_j|\xi_j|
 \leq\sqrt n\max_{k,l}|a_{kl}|,
\]
so
\[
 |A\xi|^2\leq n^2\max_{k,l}|a_{kl}|^2,
 \qquad \|A\|_{\mathrm{op}}\leq n\max_{k,l}|a_{kl}|.
\]
Then \eqref{en} yields
\[
 \|[g_{ij}](\xeuc ^\prime,t)-I\|_{\mathrm{op}}
 \leq nC_1\bigl(\sqrt {n-1}\,|\xeuc ^\prime|+t\bigr)
 \leq C_0 (|\xeuc ^\prime|+t),
\]
where $C_0 = n\sqrt {n} \,C_1$. This is \eqref{cl}. The same argument applies to ${[g_{ij}]^{-1}}$.
{    For an interior chart, the same mean-value and matrix estimates give
\[
\begin{aligned}
\|[g_{ij}](\xeuc)-I\|_{\mathrm{op}},\quad \|[g_{ij}]^{-1}(\xeuc)-I\|_{\mathrm{op}}
&\le nC_1\sum_{k=1}^n|x^k|\\
&\le n\sqrt n\,C_1|\xeuc|=C_0|\xeuc|.
\end{aligned}
\]
This proves \eqref{12}.}

\medskip

\noindent \textbf{Step 3.} Let $0<r\leq r_{\radpa} :=\min \left\{r_1, \frac{1-{\radpa}^{-1}}{4\left(1+C_0\right)}\right\}$. Then
\[
 \|[g_{ij}] -I\|_{\mathrm{op}}\leq2C_0r   
 \leq\frac{C_0}{2(1+C_0)}(1-{\radpa}^{-1})
 \leq\frac12(1-{\radpa}^{-1}).
\]

By Cauchy--Schwarz,
\[
 |\xi^T([g_{ij}] -I)\xi|
 \leq|\xi|\,|(  [g_{ij}] -I)\xi|
 \leq2C_0r   |\xi|^2.
\]

Moreover, since $1-2C_0r   \geq{\radpa}^{-1}$ and  ${\radpa}+{\radpa}^{-1}-2=\frac{({\radpa}-1)^2}{{\radpa}}\geq0$, we have
\[
 1+2C_0r   \leq2-{\radpa}^{-1}\leq{\radpa}.
\]
Thus ${\radpa}^{-1}I\leq [g_{ij}]\leq{\radpa} I$. If
$[g_{ij}](x^\prime ,t) =O\operatorname{diag}(\sigma_1,\ldots,\sigma_n)O^T$, with $O$
orthogonal, then ${\radpa}^{-1}\leq\sigma_i\leq{\radpa}
 $; consequently, $ {\radpa}^{-1}\leq\sigma_i^{-1}\leq{\radpa}$, which proves the inverse bound in \eqref{mt}. {For $0<r\le r_{x,\radpa}$, \eqref{12} gives $\|[g_{ij}]-I\|_{\mathrm{op}}\le C_0r\le(1-\radpa^{-1})/4$, and the same argument proves \eqref{mt} in the interior chart.}\end{proof}

\begin{figure}[H]
    \centering
    \includegraphics[width=0.7\textwidth]{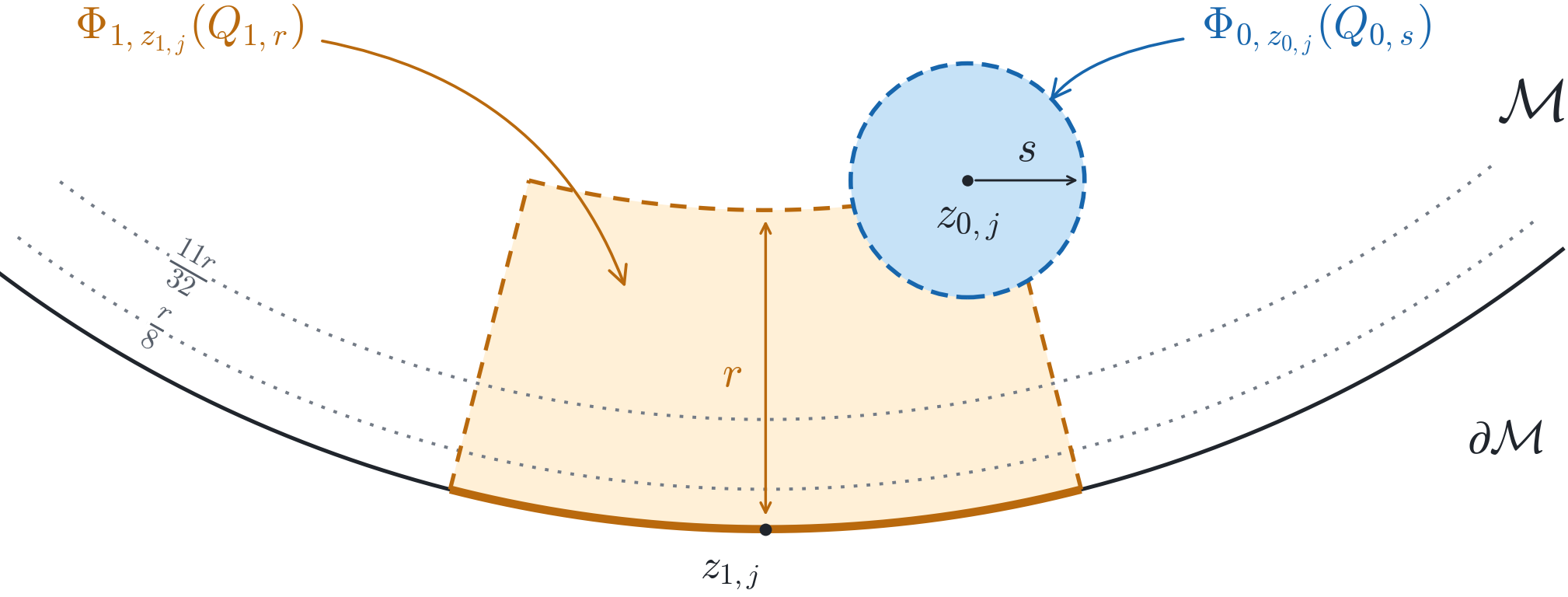}
  
    \label{fig:figura1}
\end{figure}
\begin{lemma}\label{pt}
Set
\[
 s_0:=r_2
 =\min\left\{r_1,\frac{1}{8(1+C_0)}\right\}
\]
where $r_{\radpa}$, with ${\radpa}=2$, is defined in Lemma~\ref{chs}.

Fix ${\radpa}>1$ and
\begin{equation}\label{ms}
 0<r\leq\min\{r_{\radpa},s_0/4\} \qquad  0<s\le \frac{r}{32},
\end{equation}

There exist at most countable index sets $\mathtt{J}_0,\mathtt{J}_1\subset\mathbb N$,
distinct centers $(z_{1,j}, z_{0,i}) \in \partial \cil \times \cil \setminus \partial \cil$, $(j,i)\in \mathtt{J}_1 \times \mathtt{J}_0$, and functions $\eta_{h,j}\in C_c^1( \cil )$ such that
\begin{gather}
d_g(z_{0,j} , \partial \cil) > \frac{3r}{8} \quad j\in \mathtt{J}_0 \label{21}\\
 0\leq\eta_{1,j}\leq1,\quad \operatorname{supp}\eta_{1,j} = \Phi_{1,z_{1,j}} (\overline{Q_{1,3r/4} }), \quad |\nabla_g\eta_{1,j}|\leq   {\bf c}_1 ({\radpa} , n ,  r,s) \quad j\in\mathtt{J}_1,  \label{sp}\\
 0\leq\eta_{0,j}\leq1,\quad \spt \eta_{0,j}=  \Phi_{0,z_{0,j}} (\overline{Q_{0,3s/4} }), \quad |\nabla_g\eta_{0,j}|\leq   {\bf c}_1 ({\radpa} , n ,  r,s) \quad j\in\mathtt{J}_0,  \label{47}  \\
 \sum_{j\in \mathtt{J}_1}\eta_{1,j}^q + \sum_{j\in \mathtt{J}_0}\eta_{0,j}^q =1\text{ on } \cil ,\qquad
\sum_{j\in \mathtt{J}_1}|{    \nabla_g}\eta_{1,j}|^q + \sum_{j\in \mathtt{J}_0}|{    \nabla_g}\eta_{0,j}|^q \leq   {\bf c}_2 ({\radpa} , n , q ,  r , s).\label{pw}
\end{gather}
The supports form a locally finite family in $\cil$ and there is $N_\ast>0$ such that
\begin{equation}\label{19}
\sum _{j\in \mathtt{J}_1} {\bf 1}_{\spt \eta _{1,j}}  + \sum _{j\in \mathtt{J}_0} {\bf 1}_{\spt \eta _{0,j}} \le N_\ast
 \end{equation}
 
 Moreover, we have
 \begin{equation}\label{70}
 \cil = \bigcup _{j\in \mathtt{J}_1} \Phi _{1,z_{1,j}} (Q_{1, r/2}) \cup \bigcup _{j\in \mathtt{J}_0} \Phi _{0,z_{0,j}} (Q_{0, {    s}/2}) .
 \end{equation}
\end{lemma}

\begin{proof}
\textbf{Step 1.} For $z\in\partial \cil$, define
\[
 \Phi _z ^\partial  (\xeuc ^\prime):=\exp_z^\partial({E_{1,z}}\xeuc ^\prime).
\]

Since $r_1<{\consgeo _C} $ and $s_0\leq r_1<r_{\operatorname{inj}} (\partial \cil)$ (see Definition~\ref{bg}), we obtain
\begin{equation}\label{bl}
 d_g ^{\partial}(z,\Phi _z ^\partial  (\xeuc ^\prime))=|\xeuc ^\prime|\quad(|\xeuc ^\prime|<r_1),\qquad
 B_g^\partial(z,t)=\Phi _z ^\partial  (B_t^{n-1})\quad(0<t\leq s_0).
\end{equation}

Apply \eqref{mt} with ${\radpa}=2$ at $t=0$:
\[
 \frac12 I\leq \mathtt{M}_{z}^\partial(\xeuc ^\prime)\leq2I, \qquad |\xeuc ^\prime|<s_0,
\]
where
$$
\mathtt{M}_{z}^\partial(\xeuc ^\prime):=[g_{\alpha\beta}(\xeuc ^\prime,0)]_{\alpha,\beta=1}^{n-1}.
$$

Consequently,
\[
 2^{-(n-1)}\leq\det \mathtt{M}_{z}^\partial(\xeuc ^\prime)\leq2^{n-1}.
\]
Then, by \eqref{bl},
\begin{align*}
 \operatorname{Vol}(B_g^\partial(z,s))
 &=\int_{B_{{s}}^{n-1}}\sqrt{\det \mathtt{M}_{z}^\partial(\xeuc ^\prime)}\, \din \xeuc ^\prime
\end{align*}

Thus, uniformly in $z$,
\begin{equation}\label{vl}
 c s^{n-1}\leq\operatorname{Vol}(B_g^\partial(z,{s}))\leq C s^{n-1},
 \qquad {0<s\leq s_0},
\end{equation}
where $
 c=2^{-(n-1)/2}\omega_{n-1}$ and $C=2^{(n-1)/2}\omega_{n-1}$. {Here $\omega_m$ denotes the volume of the Euclidean unit ball in $\mathbb R^m$.}
 
Similarly, if $x\in\cil\setminus\partial\cil$, we obtain
\begin{equation}\label{15}
\begin{gathered}
 d_g (x,\Phi _{0,x}  (\yeuc))=|\yeuc|, \qquad
 B_g(x,s)=\Phi _{0,x}  (B_s^n)
 \end{gathered}
\end{equation}
where $|\yeuc|,s<r_{x,2}\le\consgeo_0/2$, and $\consgeo_0$ is given in Definition~\ref{bg}.
 
Applying \eqref{mt} with ${\radpa}=2$, we have
\[
 \frac12 I\leq \mathtt{M}_{0,{x}}(\xeuc)\leq2I, \qquad |\xeuc|<r_{x,2},
\]
where
$$
\mathtt{M}_{0,{x}}(\xeuc):=[g_{ij}(\xeuc)]_{i,j=1}^{n}.
$$

Consequently,
\[
 2^{-n}\leq\det \mathtt{M}_{0,{x}}(\xeuc)\leq 2^{n}.
\]
Then, by \eqref{15},
\begin{align*}
 \vol _g(B_g(x,s))
 &=\int_{B_s^{n}}\sqrt{\det \mathtt{M}_{0,{x}} (\xeuc)}\, \din \xeuc .
\end{align*}

Thus, uniformly in $x$,
\begin{equation}\label{14}
 c s^{n}\leq\vol _g (B_g(x,s))\leq C s^{n},
 \qquad 0<s <r_{x,2},
\end{equation}
where $
 c=2^{-n/2}\omega_{n}$ and $C=2^{n/2}\omega_{n}$.

 \medskip 
 
\noindent \textbf{Step 2.} Take a countable basis of nonempty open connected sets
$(V_{ m})_{m\geq1}$ in $\partial \cil $ and choose
\[
 b_{ m}\in V_{m},\qquad
 D :=\{b_{ m}  \mid m\geq1\}.
\]
For any nonempty open $U\subset\partial \cil$, some $V_{m}$ satisfies
$V_{m}\subset U$, so $b_{m}\in U$. Therefore
\[
 \overline{D}=\partial \cil.
\]

For $r>0$, define finite sets recursively by
\begin{equation}\label{sl}
 \begin{aligned}
 P_{1}&:=  \{b_{1}\},\\
 P_{m}&:=
\left\{ \begin{aligned}
& P_{m-1}\cup\{b_{m}\}
 & &d_g ^{\partial}(b_{m},z)\geq r/4\quad\forall z\in P_{m-1},\\
 &P_{m-1} & & \text{otherwise},
 \end{aligned}\right.\\
 Z_1&:=\bigcup_{m\geq1}P_{m}.
 \end{aligned}
\end{equation}

For distinct $z,w\in Z_1$, we have
\begin{equation}\label{se}
 d_g ^{\partial}(z,w)\geq r/4.
\end{equation}

Since $Z_1\subset D$, it is at most countable. Enumerate
$Z_1$ without repetition as $(z_{1,j})_{j\in \mathtt{J}_1}$,
where $\mathtt{J}_1\subset\mathbb N$ may be finite.

We claim that
\begin{equation}\label{co}
 \partial \cil=\bigcup_{j\in \mathtt{J} _1}B_g^\partial(z_{1,j},r/2).
\end{equation}

Indeed, let $y\in\partial \cil$. By density, choose $m$ such that
\[
 d_g ^{\partial}(y,b_{m})<r/8.
\]

If $b_{m} \in Z_1$, then $b_{m} =z_{1,j}$ for some $j\in \mathtt{J}_1$ and 
\[
 d_g ^{\partial}(y,z_{1,j})<r/8<r/2.
\]
Otherwise, the failure of the condition in \eqref{sl} gives $z\in P_{m-1}$ such that
\[
 \quad d_g ^{\partial}(b_{m},z)<r/4,
\]
and hence
\[
 d_g ^{\partial}(y,z)\leq d_g ^{\partial}(y,b_{m})+d_g ^{\partial}(b_{m},z)
 <r/8+r/4=3r/8<r/2.
\]
In both cases, $y$ lies in an $r/2$-ball centered at a selected point. This proves \eqref{co}.

Similarly, for $0<s\le\min\{\frac{r}{16},\frac{1}{8(1+C_0)}\}$, there exist $\mathtt{J}_0^\prime\subset\mathbb N$ and a set 
$$
Z_0 = \bigcup _{i\in \mathtt{J}^\prime _0} \{z_{0,i}\} \subset
 \mathscr{U} :=\left\{x\in \cil \mid d_g (x, \partial \cil) > \frac{r}{8}\right\}
 $$ 
such that 
$$
{\mathscr{U}} = \bigcup _{i\in \mathtt{J} _{0}^\prime} B _{g|_{{\mathscr{U}}}} (z_{0,i} , s/2) ,
$$
where $g|_{\mathscr U}$ is the restriction of $g$ to $\mathscr U$. For $x\in\mathscr U$, we have
$$
s\le \min\left\{\frac{r}{16} , \frac{1}{8(1+C_0)}\right\} < r_{x,2} .
$$

Moreover,
\begin{equation*}
\begin{gathered}
d_{g|_{{\mathscr{U}}}}(x,y)\ge s/4\qquad\text{if }x,y\in Z_0,\ {x\ne y} .
\end{gathered}
\end{equation*}

Set
$$
\mathtt{J}_0:= \left\{ i\in \mathtt{J} _0^\prime \mid   d_g(z_{0,i},\partial \cil)  > 3r/8 \right\} .
$$

Finally, 
$$
\cil  = \bigcup _{j\in \mathtt{J} _1} \Phi _{1,z_{1,j}} \left(Q_{1 , r/2} \right) \cup \bigcup _{i\in \mathtt{J} _0 }  B_g(z_{0,i} , s/2).
$$
Observe that $B_g(z_{0,i},s/2)=B_{g|_{\mathscr U}}(z_{0,i},s/2)$ if $d_g(z_{0,i},\partial\cil)>3r/8$.

\medskip

\noindent \textbf{Step 3.} Let $0<r\le \min\{r_{\radpa} , s_0/4\} $ and  $z\in\partial \cil$. We show that 
\begin{equation}\label{4}
\# \mathtt{J}_{1,z}<\infty \qquad  \mathtt{J}_{1,z}:=\{j\in \mathtt{J}_1 \mid  d_g ^{\partial}(z,z_{1,j})<2r\} .
\end{equation}

If $j\in \mathtt{J}_{1,z}$ and $y\in B_g^\partial(z_{1,j},r/8)$, then
\[
 d_g ^{\partial}(z,y)\leq d_g ^{\partial}(z,z_{1,j})+d_g ^{\partial}(z_{1,j},y)<2r+r/8=17r/8.
\]
Therefore
\[
 \bigcup_{j\in \mathtt{J}_{1,z}}B_g^\partial(z_{1,j},r/8)\subset B_g^\partial(z,17r/8).
\]

By \eqref{se},
\begin{gather*}
B_g^\partial(z_{1,j},r/8) \cap B_g^\partial(z_{1,i},r/8) =\emptyset \quad \text { if } i\neq j ,\\
 r/8\leq s_0/32<s_0,
 \qquad 17r/8\leq17s_0/32<s_0.
\end{gather*}
Hence, for every finite $F_0\subset \mathtt{J}_{1,z}$, \eqref{vl} gives
\begin{align*}
 \#F_0\,c(r/8)^{n-1}
 &\leq\sum_{j\in F_0}\operatorname{Vol}_g(B_g^\partial(z_{1,j},r/8))\\
 &=\operatorname{Vol}_g \left(\bigcup_{j\in F_0}B_g^\partial(z_{1,j},r/8)\right)\\
 &\leq\operatorname{Vol}_g (B_g^\partial(z,17r/8))
 \leq C(17r/8)^{n-1}.
\end{align*}
Then
\[
 \#F_0\leq\frac{C}{c}
       \left(\frac{17r/8}{r/8}\right)^{n-1}
 =2^{n-1}17^{n-1}=34^{n-1}=N_1.
\]
If $\mathtt{J}_{1,z}$ contained more than $N_1$ indices, a subset of $N_1+1$ indices
would contradict this bound. Consequently,
\begin{equation}\label{ov}
 \#\mathtt{J}_{1,z}\leq N_1\qquad z\in\partial \cil .
\end{equation}
This proves \eqref{4}.

If $B_g^\partial(z,r/2)\cap B_g^\partial(z_{1,j},r)\neq\varnothing$, take a point $y$ in the
intersection. Then
\[
 d_g ^{\partial}(z,z_{1,j})\leq d_g ^{\partial}(z,y)+d_g ^{\partial}(y,z_{1,j})<r/2+r=3r/2<2r,
\]
so $j\in \mathtt{J}_{1,z}$. Thus every $B_g^\partial(z,r/2)$ meets at most $N_1$ of the balls $B_g^\partial(z_{1,j},r)$.

For $s\le r/32$, if $x\in\cil$ and $\mathtt J_{0,x}:=\{j\in\mathtt J_0\mid d_g(x,z_{0,j})<2s\}\ne\varnothing$, then
\begin{equation} \label{16}
\# \mathtt{J}_{0,x} \le N_0 \qquad  
\end{equation}
 with $N_0:=34^n$.

\medskip

\noindent \textbf{Step 4.}
Define the following functions on $\mathbb R$
\[
 \gamma(t)=
 \begin{cases}e^{-1/t},&t>0,\\0,&t\leq0,\end{cases}
 \qquad
 \beta(t)=\frac{\gamma(3/4-t)}
 {\gamma(3/4-t)+\gamma(t-1/2)}.
\]
We have $\gamma\in C^\infty(\mathbb R)$, with all derivatives zero at $0$, and
\[
 \beta\in C^\infty(\mathbb R),\qquad 0\leq\beta\leq1,
 \qquad \beta=1\ (t  \leq1/2),\qquad \beta=0\ (t  \geq3/4).
\]
Direct differentiation gives
\[
 \beta'(t)=-\frac{
 \gamma'(3/4-t  )\gamma(t -1/2)
 +\gamma(3/4-t )\gamma'(t  -1/2)}
 {\bigl(\gamma(3/4-t  )+\gamma(t  -1/2)\bigr)^2}.
\]
Since $\beta'=0$ outside $[1/2,3/4]$,
\[
C _2:=\max_{t \in[1/2,3/4]}|\beta'(t )|<\infty.
\]
The mean value theorem and
$\beta(1/2)-\beta(3/4)=1$ imply $C_2\geq4$.

For $j\in\mathtt J_i$, $i=0,1$, and $0<r\le\min\{r_{\radpa},s_0/4\}$, define 
\[
 a_{1,j}(\Phi _{1,z_{1,j}}   (\xeuc ^\prime , t )):=
 \left\{\begin{aligned}
& \beta (t/r)\beta(|\xeuc ^\prime|/r) & & \text { if }(\xeuc ^\prime ,t) \in Q_{1,r},\\
 &0 & & \text { on }  \cil\setminus \Phi_{1, z_{1, j}}(Q_{1, r}) .
 \end{aligned}
 \right.
\]
and, for $0<s\le \frac{r}{16} $, define
\[
 a_{0,j}(\Phi _{0,z_{0,j}}   (\xeuc  )):=
 \left\{\begin{aligned}
& \beta(|\xeuc |/s) & & \text{if }{\xeuc}\in Q_{0,s},\\
 &0 & & \text { on }  \cil\setminus \Phi_{0, z_{0, j}}(Q_{0, s}) .
 \end{aligned}
 \right.
\]
Here $Q_{i,r}$ and $\Phi_{i,z_{i,j}}$ are defined in \eqref{17} and \eqref{ch}.

We have
\begin{equation}\label{bp}
 \begin{gathered}
 a_{1,j}\in C_c^\infty( \cil),\quad0\leq a_{1,j}\leq1,\quad
 a_{1,j}=1\text{ on }\Phi _{1,z_{1,j}}    (\overline Q_{1,r/2} ),\\
 \operatorname{supp}a_{1,j}
 =\Phi _{1,z_{1,j}}    (\overline Q_{1,3r/4} )\Subset \Phi _{1,z_{1,j}}    ({Q_{1,r}}).
 \end{gathered}
\end{equation}
and
\begin{equation} \label{18}
 \begin{gathered}
 a_{0,j}\in C_c^\infty( \cil),\quad0\leq a_{0,j}\leq1,\quad
 a_{0,j}=1\text{ on }\Phi _{0,z_{0,j}}    (\overline Q_{0,s/2} ),\\
 \operatorname{supp}a_{0,j}
 =\Phi _{0,z_{0,j}}    (\overline Q_{0,{3s/4}} )\Subset \Phi _{0,z_{0,j}}    ({Q_{0,s}}).
 \end{gathered}
\end{equation}


For $\xeuc ^\prime\neq 0$  and $1\leq\alpha\leq n-1$,
\begin{gather*}
 \partial_\alpha(a_{1,j}\circ\Phi _{1,{z_{1,j}}}  )({\xeuc ^\prime,t})
 =\frac{\beta(t/r)\beta'(|\xeuc ^\prime|/r)}{r}\frac{x^\alpha}{|\xeuc ^\prime|},\\
 \sum_{\alpha=1}^{n-1}
 |\partial_\alpha(a_{1,j}\circ\Phi _{{1,}z_{1,j}}  )|^2
 =\frac{ \beta(t/r)^2|\beta'(|\xeuc ^\prime|/r)|^2}{r^2}\leq\frac{C_2^2}{r^2}.
\end{gather*}
At $\xeuc ^\prime=0$, all these derivatives vanish.

{For the interior cutoffs, \eqref{21} and \eqref{ms} give $s\le r/32<d_g(z_{0,j},\partial\cil)/2$, $s\le r_1/2$ and $s\le(1-\radpa^{-1})/[4(1+C_0)]$. Thus $s\le r_{z_{0,j},\radpa}$, as required for \eqref{mt}.}
Writing $v=a_{i,j}\circ\Phi_{i,z_{i,j}}$ and using \eqref{mt}, we obtain
\begin{equation}\label{bgr}
 |\nabla_ga_{i,j}|^2\circ\Phi _{i,z_{i,j}}
 =(\nabla v)^T [g_{h,k}]^{-1}\nabla v
 \leq{\radpa}|\nabla v|^2
 \leq \left\{ \begin{aligned}
 &  2{\radpa} C_2^2/r^2 &&i=1,\\
  & {\radpa} C_2^2/s^2&&i=0.
 \end{aligned} \right.
\end{equation}

\medskip

\noindent \textbf{Step 5.} Fix $x\in\cil$. If $B_g(x,1)\cap \spt a _{1,j} \neq \emptyset $, then  
\begin{align*}
d_g(x,z_{1,j})\le & d_g(x,y) + d_g(y,z_{1,j}) \\
\le & 1+ d_g^{\partial} (z_{1,j} ,z) + t \\
\le & 1+ 3r/2 ,
\end{align*}
where $y=F(z,t)\in B_g(x,1)\cap\spt a_{1,j}$. 

Set 
$$
\mathcal K_x := \left\{ z\in \partial \cil \mid d_g(x,z) \le 1+3r/2\right\} \ni z_{1,j} .
$$

For each $\bar z\in\mathcal K_x$, the set $B_g^\partial(\bar z,r/2)$ contains at most ${N_1}$ centers $z_{1,j}\in B_g^\partial(\bar z,r/2)$, by \eqref{ov}.
Compactness supplies $\bar z_1,\ldots,\bar z_{n_1}\in\mathcal K_x$ with
\[
 \mathcal K_x\subset\bigcup_{i=1}^{n_1} B_g^\partial(\bar z_{i},r/2).
\]
This implies that $B_g(x,1)$ meets at most $n_1N_1$ of the supports
$\operatorname{supp}a_{1,j}$.

If $B_g(x,1)\cap \spt a _{0,j} \neq \emptyset $, then  
\begin{align*}
d_g(x,z_{0,j})\le & d_g(x,y) + d_g(y,z_{0,j}) \\
\le & {1+3s/4\le}\ 1+3r/4 ,
\end{align*}
where $y\in B_g(x,1)\cap\spt a_{0,j}$.

We have $
z_{0,j} \in \overline{B_g(x, 1+3s/4)} $. For each $\bar x\in \overline{B_g(x, 1+3s/4)} $, the set $B_g (\bar x,s/2)$   contains at most $N_0$ centers $z_{0,j}\in B_g (\bar x,s/2)$, by {\eqref{16}}.
Compactness supplies $\bar x_1,\ldots,\bar x_{n_2}\in\overline{B_g(x, 1+3s/4)} $ with
\[
\overline{B_g(x, 1+3s/4)} \subset\bigcup_{i=1}^{n_2}{B_g}(\bar x_i,s/2).
\]
This implies that $B_g(x,1)$ meets at most $n_2N_0$ of the supports
$\operatorname{supp}a_{0,j}$.

\medskip

\noindent \textbf{Step 6.}
{At any point $x$ in a boundary support, write $x=F(z,t)$. Uniqueness of the collar coordinates shows that the same $z$ works for every boundary support containing $x$, and each corresponding center satisfies $d_g^\partial(z,z_{1,j})\le3r/4<2r$. By \eqref{ov}, there are at most $N_1$ such supports. Each interior support containing $x$ has $d_g(x,z_{0,j})\le3s/4<2s$, so \eqref{16} bounds their number by $N_0$. Consequently,
\[
 \sum_{j\in\mathtt J_1}{\bf1}_{\spt a_{1,j}}(x)
 +\sum_{j\in\mathtt J_0}{\bf1}_{\spt a_{0,j}}(x)
 \le N_0+N_1=:N_\ast.
\]
}
Using locally finite sums, define
\[
 S(x):=\left(
 \sum_{j\in\mathtt J_1}a_{1,j}(x)^q
 +\sum_{j\in\mathtt J_0}a_{0,j}(x)^q
 \right)^{1/q}.
\]

{The covering in Step~2 and \eqref{bp}, \eqref{18} imply that, at every $x\in\cil$, at least one $a_{h,j}(x)$ equals one. Thus $S(x)^q\ge1$.}
We have
\[
S(x)\geq1 ,\quad S\in C^1(\cil),\quad 1\leq S(x)
 \leq(N_0+N_1)^{1/q}.
\]

For $h=0,1$ and $j\in\mathtt J_h$, define
\[
\eta_{h,j}(x):=\frac{a_{h,j}(x)}{S(x)}.
\]

Then
\[
\begin{gathered}
 \eta_{h,j}\in C_c^1(\cil),\quad
 0\leq\eta_{h,j}\leq a_{h,j}\leq1,\\
 \operatorname{supp}\eta_{h,j}
 =\operatorname{supp}a_{h,j},\\
 \operatorname{supp}\eta_{1,j}\Subset
 \Phi_{1,z_{1,j}}(Q_{1,r}),\\
 \operatorname{supp}\eta_{0,j}\Subset
 \Phi_{0,z_{0,j}}(Q_{0,s}).
\end{gathered}
\]

Moreover,
\[
\begin{aligned}
 \sum_{j\in\mathtt J_1}\eta_{1,j}^q
 +\sum_{j\in\mathtt J_0}\eta_{0,j}^q
 &=S^{-q}\left(\sum_{j\in\mathtt J_1}a_{1,j}^q
 +\sum_{j\in\mathtt J_0}a_{0,j}^q\right)\\
 &=S^{-q}S^q=1.
\end{aligned}
\]
This proves the first equality in \eqref{pw}.

\medskip

\noindent \textbf{Step 7.}
We complete the proof of \eqref{sp} and \eqref{pw}. The chain rule and local finiteness give
\[
 \nabla_g S
 =S^{1-q}\left(
 \sum_{j\in\mathtt J_1}a_{1,j}^{q-1}\nabla_ga_{1,j}
 +\sum_{j\in\mathtt J_0}a_{0,j}^{q-1}\nabla_ga_{0,j}
 \right).
\]
Using $S\ge1$, $0\le a_{h,j}\le1$, and $q>1$,
\[
\begin{aligned}
 |\nabla_g S|
 &\leq S^{1-q}\sum_{h=0}^1\sum_{j\in\mathtt J_h}
 a_{h,j}^{q-1}|\nabla_g a_{h,j}|\\
 &\leq\sum_{h=0}^1\sum_{j\in\mathtt J_h}
 |\nabla_g a_{h,j}|\\
 &\leq N_1\frac{\sqrt{2{\radpa}}C_2}{r}
 +N_0\frac{\sqrt{\radpa} C_2}{s}.
\end{aligned}
\]
For each $(h,j)$,
\[
 \nabla_g\eta_{h,j}
 =\frac{\nabla_ga_{h,j}}{S}
 -\frac{a_{h,j}}{S^2}\nabla_g S.
\]
Consequently, for both families,
\[
\begin{aligned}
 |\nabla_g\eta_{h,j}|
 &\leq|\nabla_ga_{h,j}|+|\nabla_gS|\\
 &\leq C_2\sqrt{\radpa}\left(
 \frac{\sqrt2(N_1+1)}{r}+\frac{N_0+1}{s}
 \right)\\
 &=:{\bf c}_1({\radpa},n,{r,s}),
 \qquad (h=0,1).
\end{aligned}
\]

Finally, by \eqref{ov} and \eqref{16}, for each $x\in\cil$ there are at most
$N_0+N_1$ indices with $a_{h,j}(x)\ne0$. {The same bound holds for indices with $x\in\spt a_{h,j}$, as shown above.}
Outside the support of $a_{h,j}$,
$\nabla_g\eta_{h,j}=0$; hence,
\[
\begin{aligned}
 \sum_{j\in\mathtt J_1}|\nabla_g\eta_{1,j}(x)|^q
 +\sum_{j\in\mathtt J_0}|\nabla_g\eta_{0,j}(x)|^q
 \leq  {\bf c}_2({\radpa},n,q,{r,s}),
\end{aligned}
\]
where ${\bf c}_2 := (N_0+N_1)
 \left[\mathbf c_1({\radpa},n,{r,s})\right]^q$.\end{proof}

\subsection{Sobolev and trace estimates}
Throughout this subsection, we assume that
\[
 n\geq2,\qquad 1<q<n .
\]

Here $C_c^k({\cil})$ consists of functions of class $C^k$ up to $\partial {\cil}$
with compact support in ${\cil}$. Their boundary values may be nonzero.
The same convention applies to $C_c^k(\overline{\mathbb R^n_+})$.

We use the following constants throughout this article (see \cite{MaggiNeumayer2017,zbMATH05077944}):
\begin{gather}
 \consobb     :=\sup_{\substack{v\in C_c^\infty(\overline{\mathbb R^n_+})\\v\neq0}}
 \frac{\displaystyle
  \left(\int_{\mathbb R^{n-1}}|v(\xeuc ^\prime,0)|^{\exnum}  \,\din \xeuc ^\prime\right)^{q/{\exnum}}}
 {\displaystyle\int_{\mathbb R^n_+}|\nabla v|^q\,\din \xeuc}
 \in(0,\infty),\label{k}\\
 \consobbm    =\sup_{\substack{v\in C_c^\infty( \overline{\mathbb R^n_+}  )\\v\neq0}}
 \frac{\displaystyle
  \left(\int_{\mathbb R^{n }_+}|v |^{\exast}  \,\din \xeuc \right)^{q/{\exast}}}
 {\displaystyle\int_{\mathbb R^n_+}|\nabla v|^q\,\din \xeuc}
 \in(0,\infty), \label{44}
\end{gather}
and
\begin{equation}\label{25}
 \consob    =\sup_{\substack{v\in C_c^\infty( \mathbb R^n  )\\v\neq0}}
 \frac{\displaystyle
  \left(\int_{\mathbb R^{n }}|v |^{\exast}  \,\din \xeuc \right)^{q/{\exast}}}
 {\displaystyle\int_{\mathbb R^n}|\nabla v|^q\,\din \xeuc}
 \in(0,\infty).
\end{equation}
 
We have $\consobbm = 2^{q/n } \consob$.
{By smooth approximation, the inequalities in \eqref{k}, \eqref{44} and \eqref{25} also hold for compactly supported $C^1$ functions on their respective domains.}


Under Definition~\ref{bg}, we define
\[
W^{1,q}(\cil)
:=
\overline{C_c^\infty(\cil)}
^{\,\|\cdot\|_{W^{1,q}(\cil)}} ,
\]
where
\[
 \|u\|_{W^{1,q}({\cil})}^q
 :=\int_{{\cil}}\bigl(|u|^q+|\nabla_gu|^q\bigr)\,\din v_g .
\]
Equivalently, by Sobolev density on complete
Riemannian manifolds with smooth boundary,
\[
W^{1,q}(\cil)
=
\left\{
u\in L^q(\cil)\mid
\nabla_g u\in L^q({T\cil})
\right\},
\]
where $\nabla_g u$ is understood in the
distributional sense on the interior of $\cil$.



To prove the following results, we use the technique in
\cite[Theorem 7.1]{zbMATH01447265}.  
\begin{proposition}\label{pr}
Under the assumptions of Definition~\ref{bg}, the restriction map admits a unique bounded
linear extension
\begin{equation}\label{tr}
 \operatorname{Tr}:W^{1,q}({\cil})\to L^{\exnum}(\partial {\cil},\din s_g ),
 \qquad
 \operatorname{Tr}u=u|_{\partial {\cil}}\quad(u\in C_c^\infty({\cil})).
\end{equation}
For every $\epsilon>0$, there exists $C_\epsilon>0$ such that
\begin{equation}\label{iq}
  \begin{aligned}
 &\left(\int_{\partial {\cil}}|\operatorname{Tr}u|^{\exnum}\,\din s_g \right)^{q/{\exnum}}\\
 &\qquad\leq(\consobb     +\epsilon)\int_{{\cil}}|\nabla_gu|^q\,\din v_g 
              +C_\epsilon\int_{\cil}|u|^q\,\din v_g ,
 \qquad u\in W^{1,q}({\cil}).
 \end{aligned}
\end{equation}

\end{proposition}

The proof of the following proposition is similar to that of Proposition~\ref{pr} and is therefore omitted.
\begin{proposition}\label{26}
Under the assumptions of Definition~\ref{bg}, for every $\epsilon>0$, there exists $C_\epsilon>0$ such that
\begin{equation}\label{27}
  \begin{aligned}
 &\left(\int_{ {\cil}}|u|^{\exast}\,{ \din v_g} \right)^{q/\exast}\\
 &\qquad\leq({\consobbm}+\epsilon)\int_{{\cil}}|\nabla_gu|^q\,\din v_g 
              +C_\epsilon\int_{\cil}|u|^q\,\din v_g ,
 \qquad u\in W^{1,q}({\cil}).
 \end{aligned}
\end{equation}
\end{proposition}


%
\begin{lemma}\label{lc}
Let ${\radpa}>1$, $0<r\leq r_{\radpa}$, {$z\in\partial\cil$}, and
$w\in C_c^1({\cil})$ with $\operatorname{supp}w\Subset {\Phi_{1,z}(Q_{1,r})}$.
The support may intersect $\partial {\cil}$. Then
\begin{equation}\label{li}
 \left(\int_{\partial {\cil}}|w|^{\exnum}\,\din s_g \right)^{q/{\exnum}}
 \leq\consobb     {\radpa}^n\int_{\cil}|\nabla_gw|^q\,\din v_g .
\end{equation}
\end{lemma}

\begin{proof}
 In the coordinates given by $\Phi _{1,z}$ and $\Phi_z ^\partial  : B_{r_1}^{n-1} \to \partial \cil$, with $\Phi_z^\partial\left(x^{\prime}\right):=F\left(\exp _z^{\partial \cil}\left({E_{1,z}}x^{\prime}\right), 0\right)$,
 set
\[
 J_z=\sqrt{\det [g_{ij}]},
 \qquad J_{z} ^\partial=\sqrt{\det [g_{\alpha \beta} ^\partial]}.
\]

By \eqref{mt}, we have
\begin{equation} \label{22}
 {\radpa}^{-n}\leq\det [g_{ij}]\leq{\radpa}^n,
 \qquad {\radpa}^{-n/2}\leq J_z\leq{\radpa}^{n/2}.
\end{equation}
For $\xi'\in\mathbb R^{n-1}$, apply \eqref{mt} to $(\xi',0)$:
\[
 {\radpa}^{-1}|\xi'|^2\leq(\xi')^T[g_{\alpha \beta} ^\partial]\xi'
 \leq{\radpa}|\xi'|^2.
\]

It follows that
\begin{equation}\label{dn}
 {\radpa}^{-(n-1)/2}\leq J_{z} ^\partial\leq{\radpa}^{(n-1)/2},
 \qquad \din v_g =J_z\,\din \xeuc ^\prime\,\din t,\qquad
 \din s_g =J_{z} ^\partial\, \din \xeuc ^\prime.
\end{equation}

Let $v=w\circ{\Phi_{1,z}}$ on $Q_{{1},r}$, extended by zero to $\overline{\mathbb R^n_+}$; we have
\[
 g(\nabla_gw , d{\Phi_{1,z}}(  e_i))=\frac{\din v}{\din x_i} \quad i=1, \ldots n
\]
and 
\begin{equation}\label{gr}
 |\nabla_gw|^2\circ{\Phi_{1,z}}
 =(\nabla v)^T[g_{ij}]^{-1}\nabla v\geq{\radpa}^{-1}|\nabla v|^2.
\end{equation}

Hence,
\begin{align}
 \int_{\cil}|\nabla_gw|^q\,\din v_g 
 &=\int_{Q_{{1},r}}\left[(\nabla v)^T[g_{ij}]^{-1}\nabla v\right]^{q/2}J_z\, \din \xeuc ^\prime\,\din t
 \nonumber\\
 &\geq{\radpa}^{-(n+q)/2}
          \int_{\mathbb R^n_+}|\nabla v|^q\, \din \xeuc .\label{eg}
\end{align}

 By \eqref{dn} and \eqref{eg},
\begin{align*}
 \left(\int_{\partial {\cil}}|w|^{\exnum}\,\din s_g \right)^{q/{\exnum}}
 &=\left(\int_{B_r^{n-1}}|v(\xeuc ^\prime,0)|^{\exnum}J_{z} ^\partial(\xeuc ^\prime)\,\din \xeuc ^\prime\right)^{q/{\exnum}}\\
 &\leq{\radpa}^{\frac{(n-1)q}{2{\exnum}}}
       \left(\int_{\mathbb R^{n-1}}|v(\xeuc ^\prime,0)|^{\exnum}\, \din \xeuc ^\prime\right)^{q/{\exnum}}\\
 &\leq\consobb     {\radpa}^{\frac{(n-1)q}{2{\exnum}}}
          \int_{\mathbb R^n_+}|\nabla v|^q\, \din \xeuc\\
 &\leq\consobb     {\radpa}^{\frac{(n-1)q}{2{\exnum}}+\frac{n+q}{2}}
          \int_{\cil}|\nabla_gw|^q\,\din v_g .
\end{align*}
Since ${\exnum}=q(n-1)/(n-q)$,
\[
 \frac{(n-1) q}{2{\exnum}}=\frac{n-q}{2},\qquad
 \frac{(n-1)q}{2{\exnum}}+\frac{n+q}{2}=n.
\]
This proves \eqref{li}.\end{proof}

 
\begin{lemma}\label{ls}
Let ${\radpa}>1$ and $w\in C_c^1(\cil)$.  

If $\operatorname{supp}w\Subset\Phi_{1,z}(Q_{1,r}),
 z\in\partial\cil$, $0<r\leq r_{\radpa}$, then
\begin{equation}\label{o2}
 \left(\int_\cil |w|^{\exast}\,\din v_g\right)^{q/\exast}
 \leq\consobbm    {\radpa}^n
 \int_\cil |\nabla_gw|^q\,\din v_g.
\end{equation}

If $\operatorname{supp}w\Subset\Phi_{0,z}(Q_{0,r}),
 z\in\cil\setminus\partial\cil,\quad 0<r\leq r_{z,{\radpa}}$, then
\begin{equation}\label{46}
 \left(\int_\cil |w|^{\exast}\,\din v_g\right)^{q/\exast}
 \leq\consob    {\radpa}^n
 \int_\cil |\nabla_gw|^q\,\din v_g.
\end{equation}
 
\end{lemma}

\begin{proof}

For either admissible chart, let $v=w\circ\Phi_{i,z}$,
extended by zero outside $Q_{i,r}$ (to $\mathbb R^n_+$ for
$i=1$, and to $\mathbb R^n$ for $i=0$). By \eqref{mt},
\[
 {\radpa}^{-1}I\leq[g_{ij}]\leq{\radpa} I,
 \qquad
 {\radpa}^{-1}I\leq[g_{ij}]^{-1}\leq{\radpa} I.
\]
Writing $J_z=\sqrt{\det[g_{ij}]}$ as in Lemma~\ref{lc},
\begin{equation}\label{45}
 {\radpa}^{-n/2}\leq J_z\leq{\radpa}^{n/2},
 \qquad
 |\nabla_gw|^q\circ\Phi_{i,z}
 =\bigl((\nabla v)^T[g_{ij}]^{-1}\nabla v\bigr)^{q/2}
 \geq{\radpa}^{-q/2}|\nabla v|^q.
\end{equation}
Thus
\[
 \left(\int_\cil|w|^{\exast}\,\din v_g\right)^{q/\exast}
 \leq{\radpa}^{nq/(2\exast)}
 \left(\int_{Q_{i,r}}|v|^{\exast}\,\din\xeuc\right)^{q/\exast}
 ={\radpa}^{(n-q)/2}
 \left(\int_{Q_{i,r}}|v|^{\exast}\,\din\xeuc\right)^{q/\exast}
 \]
 and
 \[
{\radpa}^{-(n+q)/2}
 \int_{Q_{i,r}}|\nabla v|^q\,\din\xeuc \le  \int_\cil|\nabla_gw|^q\,\din v_g.
\]

For the case $\spt {w}\Subset\Phi_{i,z}(Q_{i,r})$ with $i=1$, we obtain
\[
 \begin{aligned}
 \left(\int_\cil|w|^{\exast}\,\din v_g\right)^{q/\exast}
 &\leq{\radpa}^{(n-q)/2}
 \left(\int_{\mathbb R^n_+}|v|^{\exast}\,\din\xeuc\right)^{q/\exast}\\
 &\leq\consobbm   {\radpa}^{(n-q)/2}
 \int_{\mathbb R^n_+}|\nabla v|^q\,\din\xeuc\\
 &\leq\consobbm    {\radpa}^{(n-q)/2+(n+q)/2}
 \int_\cil|\nabla_gw|^q\,\din v_g\\
 &=\consobbm    {\radpa}^n\int_\cil|\nabla_gw|^q\,\din v_g.
 \end{aligned}
\]

For the case $i=0$, using the coordinates on $\Phi _{i,z}(Q_{i,r})$, we obtain
{
\[
 \begin{aligned}
 \left(\int_\cil|w|^{\exast}\,\din v_g\right)^{q/\exast}
 &\le\radpa^{(n-q)/2}\left(\int_{\mathbb R^n}|v|^{\exast}\,\din\xeuc\right)^{q/\exast}\\
 &\le\consob\radpa^{(n-q)/2}\int_{\mathbb R^n}|\nabla v|^q\,\din\xeuc\\
 &\le\consob\radpa^{(n-q)/2+(n+q)/2}\int_\cil|\nabla_gw|^q\,\din v_g
 =\consob\radpa^n\int_\cil|\nabla_gw|^q\,\din v_g.
 \end{aligned}
\]
}\end{proof}




%

\begin{proof}[Proof of Proposition \ref{pr}]
\textbf{Step 1.}
Fix $\epsilon>0$ and set
\begin{equation}\label{ce}
 {\radpa}=\left(1+\frac{\epsilon}{\consobb     }\right)^{1/(n+1)}>1,
 \qquad r=\min\{r_{\radpa},s_0/4\},\qquad
 \consobb     {\radpa}^{n+1}=\consobb     +\epsilon.
\end{equation}
{In applying Lemma~\ref{pt} with the parameters in \eqref{ce}, choose $s:=r/32>0$. Then \eqref{ms} holds, and both radii are fixed by $\epsilon$, $n$, $q$, and the geometry.}

Take the cutoffs of Lemma~\ref{pt}. Let $u\in C_c^1({\cil})$. Only finitely many $\eta_{1,j}u$ are nonzero. On $\partial {\cil}$, \eqref{pw} gives
\[
 |u|^q=\sum_{h=0}^1\sum_{j\in\mathtt J_h} \eta_{h,j}^q|u|^q=\sum_j|\eta_{1,j}u|^q.
\]

Because ${\exnum}/q>1$,
\begin{align}
 \|u\|_{L^{\exnum} (\partial {\cil})}^q
 &
 =\big\||u|^q\big\|_{L^{{\exnum}/q} (\partial {\cil})}\nonumber\\
 &=\Big\|\sum_j|\eta_{1,j}u|^q\Big\|_{L^{{\exnum}/q} (\partial {\cil})}
 \leq\sum_j\big\||\eta_{1,j}u|^q\big\|_{L^{{\exnum}/q} (\partial {\cil})}\nonumber\\
 &=\sum_j\|\eta_{1,j}u\|_{L^{\exnum}(\partial {\cil})}^q
 \leq\consobb     {\radpa}^n
       \sum_j\int_{\cil}|\nabla_g(\eta_{1,j}u)|^q\,\din v_g ,\label{su}
\end{align}
by Lemma~\ref{lc}.

\medskip

\noindent \textbf{Step 2.}
Set
\begin{equation}\label{yc}
 a={\radpa}^{-1/(q-1)}\in(0,1),\qquad
 c_{\radpa}=(1-a)^{1-q}
 =\left[1-{\radpa}^{-1/(q-1)}\right]^{1-q}.
\end{equation}

For vectors $X,Y$, convexity of $|\cdot|^q$
gives
\begin{align}
 |X+Y|^q
 &=\left|a\frac{X}{a}+(1-a)\frac{Y}{1-a}\right|^q\nonumber\\
 &\leq a\left|\frac{X}{a}\right|^q
       +(1-a)\left|\frac{Y}{1-a}\right|^q\nonumber\\
 &=a^{1-q}|X|^q+(1-a)^{1-q}|Y|^q
 ={\radpa}|X|^q+c_{\radpa}|Y|^q.\label{cx}
\end{align}

Since $\nabla_g(\eta_{1,j}u)=\eta_{1,j}\nabla_gu+u\nabla_g\eta_{1,j}$,
\begin{align*}
 \sum_j|\nabla_g(\eta_{1,j}u)|^q
 &\leq{\radpa}|\nabla_gu|^q\sum_j\eta_{1,j}^q
       +c_{\radpa}|u|^q\sum_j|\nabla_g\eta_{1,j}|^q\\
 &\leq{\radpa}|\nabla_gu|^q+c_{\radpa} {\bf c}_2 |u|^q.
\end{align*}
Insert this bound into \eqref{su} and use \eqref{ce}:
\begin{align}
 \|u\|_{L^{\exnum} (\partial {\cil})}^q
 &\leq\consobb     {\radpa}^{n+1}\int_{\cil}|\nabla_gu|^q\,\din v_g 
       +\consobb     {\radpa}^nc_{\radpa} {\bf c}_2    \int_{\cil}|u|^q\,\din v_g \nonumber\\
 &=(\consobb     +\epsilon)\int_{\cil}|\nabla_gu|^q\,\din v_g 
       +C_\epsilon\int_{\cil}|u|^q\,\din v_g , \label{5}
\end{align}
where $C_\epsilon=\consobb     {\radpa}^n
       (1-{\radpa}^{-1/(q-1)})^{1-q}{\bf c}_2$. This proves \eqref{iq} for $u\in C_c^1({\cil})$.

\medskip

\noindent \textbf{Step 3.} For every $u\in W^{1,q}({\cil})$, there exists a sequence
$u_m\in C_c^\infty({\cil})$ with $u_m\to u$ in $W^{1,q}({\cil})$.
Applying \eqref{5} with $\epsilon=1$ to $u_m-u_l$, we obtain
\[
 \|u_m-u_l\|_{L^{\exnum}(\partial {\cil})}^q
 \leq(\consobb     +1)\|\nabla_g(u_m-u_l)\|_{L^q ({\cil})}^q
          +C_1\|u_m-u_l\|_{L^q ({\cil})}^q\to0.
\]

Since the space $L^{\exnum}(\partial {\cil})$ is complete, we define
\begin{equation}\label{lm}
 \operatorname{Tr}u:=\lim_{m\to\infty}u_m|_{\partial {\cil}}
 \quad\text{in }L^{\exnum}(\partial {\cil}).
\end{equation}
For another approximation $v_m\to u$, the same estimate on $u_m-v_m$
shows that the two limits agree.  Passing to the limit in \eqref{5} with $\epsilon=1$ gives
\[
 \|\operatorname{Tr}u\|_{L^{\exnum} (\partial {\cil})}^q
 \leq(\consobb     +1)\|\nabla_gu\|_{L^q (\cil)}^q+C_1\|u\|_{L^q ({\cil})}^q.
\]
Thus the operator in \eqref{tr} is bounded. Density gives its uniqueness.


Finally, passing to the limit under strong convergence, we conclude that \eqref{iq} holds for every $u\in W^{1,q}(\cil)$.
\end{proof}

\begin{lemma} \label{23}
Assume that $n\ge 2$ {    and the conditions of Definition~\ref{bg} hold}. For $v\in C^1_c(\cil)$ and $\epsilon _0>0$, we have the estimates
\begin{gather}
\|v\|_{L^2(\partial \cil)}^2 \leq C_1 \left( \int_{\cil}{|v|\,|\nabla_g v|} \, \din v_g+\int_{\cil}|v|^2 \, \din v_g\right)\label{54}\\
 \|   v\|_{L^2(\partial\cil)}^2
 \le\epsilon _0 \|\nabla_gv\|_{L^2(\cil)}^2
 +C_2(1+\epsilon _0 ^{-1})\|v\|_{L^2(\cil)}^2\label{n3},
\end{gather}
where $C_i=C_i(n,g)>0$, $i=1,2$.

\end{lemma}
\begin{proof}
Fix $v\in C_c^1(\cil)$ and $\epsilon _0>0$. Applying Lemma~\ref{pt} with $q=2$, $\radpa=2$, $r=s_0/4$, and $s=r/32$, we obtain
\[
\sum_{j\in\mathtt J_1}\eta_{1,j}^2
=\sum_{j\in\mathtt J_1}\eta_{1,j}^2
 +\sum_{j\in\mathtt J_0}\eta_{0,j}^2=1{    \qquad\text{on }\partial\cil},
\]
since the interior cutoffs vanish on $\partial\cil$. The sums are locally finite, and $v$ has compact support. Consequently,
\begin{equation}\label{34}
\begin{aligned}
\|v\|_{L^2(\partial\cil)}^2
=\sum_{j\in\mathtt J_1}
  \|v\eta_{1,j}\|_{L^2(\partial\cil)}^2.
\end{aligned}
\end{equation}

For $j\in\mathtt J_1$ put
$u=(v\eta_{1,j})\circ\Phi_{1,z_{1,j}}$.
{Extend $u$ by zero for $t\ge r$; by \eqref{sp}, it already vanishes for $3r/4\le t<r$.}
By \eqref{sp}, 
\[
\begin{aligned}
|u(\xeuc',0)|^2
&=-\int_0^r\partial_t(|u(\xeuc',t)|^2)\,\din t\\
&=-2\int_0^r u(\xeuc',t)\partial_tu(\xeuc',t)\,\din t\\
&\le2\int_0^r |u|\,|\nabla u|\,\din t.
\end{aligned}
\]

By \eqref{mt}, \eqref{22}, and \eqref{dn},
\[
\begin{gathered}
\radpa^{-n/2}\le J_j\le\radpa^{n/2},\qquad
\radpa^{-(n-1)/2}\le J_j^\partial\le\radpa^{(n-1)/2},\\
|\nabla u|\le\radpa^{1/2}
 |\nabla_g(v\eta_{1,j})|\circ\Phi_{1,z_{1,j}}.
\end{gathered}
\]
Hence
\[
\begin{aligned}
\int_{\partial\cil}|v\eta_{1,j}|^2\,\din s_g
&=\int_{B_r^{n-1}}|u(\xeuc',0)|^2J_j^\partial\,\din\xeuc'\\
&\le2\radpa^{(n-1)/2}\int_{Q_{1,r}}|u|\,|\nabla u|\,\din\xeuc'\,\din t\\
&\le2\radpa^{(n-1)/2+1/2}\int_{Q_{1,r}}
 |u|\,|\nabla_g(v\eta_{1,j})|\circ\Phi_{1,z_{1,j}}
 \,\din\xeuc'\,\din t\\
&\le2\radpa^{(n-1)/2+1/2+n/2}
 \int_{\cil}|v\eta_{1,j}|
 |\nabla_g(v\eta_{1,j})|\,\din v_g\\
&=2\radpa^n\int_{\cil}|v\eta_{1,j}|
 |\nabla_g(v\eta_{1,j})|\,\din v_g.
\end{aligned}
\]

By \eqref{34} and the preceding estimate,
\[
\begin{aligned}
\|v\|_{L^2(\partial\cil)}^2
&\le2\radpa^n\sum_{j\in\mathtt J_1}
 \int_{\cil}|v|\eta_{1,j}
 |\eta_{1,j}\nabla_gv+v\nabla_g\eta_{1,j}|\,\din v_g\\
&\le2\radpa^n\int_{\cil}|v||\nabla_gv|
 \sum_{j\in\mathtt J_1}\eta_{1,j}^2\,\din v_g\\
&\quad+2\radpa^n\int_{\cil}|v|^2
 \sum_{j\in\mathtt J_1}\eta_{1,j}|\nabla_g\eta_{1,j}|\,\din v_g.
\end{aligned}
\]

Using \eqref{pw} and Cauchy--Schwarz,
\[
\sum_{j\in\mathtt J_1}\eta_{1,j}|\nabla_g\eta_{1,j}|
\le\left(\sum_j\eta_{1,j}^2\right)^{1/2}
 \left(\sum_j|\nabla_g\eta_{1,j}|^2\right)^{1/2}
\le\sqrt{\mathbf c_2}.
\]

It follows that
\[
\begin{aligned}
\|v\|_{L^2(\partial\cil)}^2
&\le2\radpa^n\int_{\cil}|v||\nabla_gv|\,\din v_g
 +2\radpa^n\sqrt{\mathbf c_2}
 \int_{\cil}|v|^2\,\din v_g .
\end{aligned}
\]
This proves \eqref{54}, {with $C_1:=2\radpa^n\max\{1,\sqrt{\mathbf c_2}\}$}.

We also have  
\[
\begin{aligned}
\|v\|_{L^2(\partial\cil)}^2
&\le\epsilon _0\|\nabla_gv\|_{L^2(\cil)}^2
 +\left(\frac{\radpa^{2n}}\epsilon _0
 +2\radpa^n\sqrt{\mathbf c_2}\right)
 \|v\|_{L^2(\cil)}^2,
\end{aligned}
\]
which proves \eqref{n3}.\end{proof}

\section{Variational levels and critical minimizers}\label{r1}

 {Throughout this section, $\consob$, $\consobb$, and $\consobbm$ are evaluated at $q=2$.} Fix $\centro\in\cil$, $j\in\mathbb N$, and $\exa_j:=B_g(\centro,j)$.

The space $C_c^\infty(\exa_j)$ consists of functions
smooth up to $\exa_j\cap\partial\cil$,
with supports compactly contained in $\exa_j$. Define
\begin{equation}\label{w1}
 W^{1,2}_{\partial}(\exa_j)
 :=\overline{C_c^\infty(\exa_j)}^{\,W^{1,2}(\exa_j)}.
\end{equation}
{    We identify each $u\in W^{1,2}_{\partial}(\exa_j)$ with its zero extension to $\cil$. The zero extensions of an approximating sequence in $C_c^\infty(\exa_j)$ preserve the $W^{1,2}$ norms and converge in $W^{1,2}(\cil)$ to that extension.}
\begin{lemma}\label{l2} If $\consy\geq 0$, then the following statements hold:
\begin{enumerate}[label=$(\roman*)$] 
\item  \label{89}  $ 0<\Lambda^+_{n,\consy}<\infty$.

\item \label{73} If  $Y_{\consy,\infty}(\cil)>0$, then $-\infty<Y_{\consy}(\cil)$.

\item \label{74} ${Y}_{\consy}(\cil)
 \leq {Y}_{\consy,\infty}(\cil)
 \leq\Lambda^+_{n,\consy}$.

\end{enumerate}
\end{lemma}

\begin{proof}

{\it Proof of \ref{89}.} By \eqref{k} and \eqref{44} with $q=2$, we have
\begin{align*}
 &\left(\|v\|_{L^{{\dcrit}}(\mathbb R^n_+)}^{{\dcrit}}
 +\consy\|v(\cdot,0)\|_{L^{{\dcritb}}(\mathbb R^{n-1})}^{{\dcrit}}\right)^{2/{\dcrit}}\\
 &\qquad\leq\|v\|_{L^{{\dcrit}}(\mathbb R^n_+)}^2
 +\consy^{2/{\dcrit}}\|v(\cdot,0)\|_{L^{{\dcritb}}(\mathbb R^{n-1})}^2\\
 &\qquad\leq
 \left(\consobbm+ \consy^{(n-2)/n}\consobb\right)
 \int_{\mathbb R^n_+}|\nabla v|^2\,\din x.
\end{align*}
Then
\[
 0<\left(\consobbm
 +\consy^{(n-2)/n}\consobb\right)^{-1}
 \leq\Lambda^+_{n,\consy}<\infty.
\]

\medskip

\noindent {\it Proof of \ref{73}.} Since $Y_{\consy,\infty}(\cil)>0$, choose $L\in\mathbb N$ sufficiently large that
\[
Y_{\consy}(\cil\setminus\overline{\exa_L})>0.
\]

Choose $\psi\in C_c^\infty(\exa_{L+1})$ with $0\le\psi\le\pi/2$ and $\psi=\pi/2$ on a neighborhood of $\overline{\exa_L}$. Set
\[
\chi_0=\sin\psi,\qquad \chi_1=\cos\psi.
\]
Then $\chi_0$ has compact support in $\exa_{L+1}$, $\chi_1=0$ near $\overline{\exa_L}$, and
\[
\begin{gathered}
\chi_0^2+\chi_1^2=1,\qquad
\chi_0\nabla_g\chi_0+\chi_1\nabla_g\chi_1=0,\\
|\nabla_g\chi_0|^2+|\nabla_g\chi_1|^2=|\nabla_g\psi|^2.
\end{gathered}
\]

For $u\in C_c^\infty(\cil)$, we have
\[
\begin{aligned}
\sum_{i=0}^1|\nabla_g(\chi_i u)|^2
&=\sum_{i=0}^1\left(\chi_i^2|\nabla_gu|^2
+u^2|\nabla_g\chi_i|^2
+2\chi_i u\langle\nabla_gu,\nabla_g\chi_i\rangle_g\right)\\
&=|\nabla_gu|^2+u^2|\nabla_g\psi|^2.
\end{aligned}
\]
Therefore,
\begin{equation}\label{86}
E_g(u)=E_g(\chi_0u)+E_g(\chi_1u)
-\int_{\cil}u^2|\nabla_g\psi|^2\,\mathrm{d}v_g.
\end{equation}
We now estimate the right-hand side. By the definition of the exterior level,
\begin{equation}\label{87}
E_g(\chi_1u)\ge
Y_{\consy}(\cil\setminus\overline{\exa_L})
\opnj(\chi_1u)^{2/{\dcrit}}\ge0.
\end{equation}
{    This inequality is immediate if $\chi_1u\equiv0$.}

Moreover, the bounded geometry assumptions imply that $$
C_1:=\frac{n-2}{4(n-1)}\|R_g^-\|_{L^\infty(\cil)}, \qquad  C_2:=\frac{n-2}{2(n-1)}\|H_g^-\|_{L^\infty(\partial\cil)}
$$
are finite. Taking $\epsilon_0=\frac{1}{2(C_2+1)}$ in Lemma~\ref{23}, we obtain
\begin{equation}\label{88}
\begin{aligned}
E_g(\chi_0u)
&\ge\|\nabla_g(\chi_0u)\|_{L^2(\cil)}^2
-C_1\|\chi_0u\|_{L^2(\cil)}^2
-C_2\|\chi_0u\|_{L^2(\partial\cil)}^2\\
&\ge\frac12\|\nabla_g(\chi_0u)\|_{L^2(\cil)}^2
-C_3\|\chi_0u\|_{L^2(\cil)}^2\\
&\ge-C_3\int_{\exa_{L+1}}u^2\,\mathrm{d}v_g,
\end{aligned}
\end{equation}
where $C_3>0$ is finite and independent of $u$. Combining \eqref{86}--\eqref{88} yields
\[
E_g(u)\ge
-\left(C_3+\|\nabla_g\psi\|_{L^\infty(\cil)}^2\right)
\int_{\exa_{L+1}}u^2\,\mathrm{d}v_g.
\]

By H\"older's inequality and the definition of $\opnj$,
\[
\begin{aligned}
\int_{\exa_{L+1}}u^2\,\mathrm{d}v_g
&\le\operatorname{Vol}_g(\exa_{L+1})^{2/n}
\left(\int_{\exa_{L+1}}|u|^{{\dcrit}}\,\mathrm{d}v_g\right)^{2/{\dcrit}}\\
&\le\operatorname{Vol}_g(\exa_{L+1})^{2/n}
\|u\|_{L^{{\dcrit}}(\cil)}^2\\
&\le\operatorname{Vol}_g(\exa_{L+1})^{2/n}
\opnj(u)^{2/{\dcrit}}.
\end{aligned}
\]
Thus, for every $u\in C_c^\infty(\cil){    \setminus\{0\}}$,
\[
\frac{E_g(u)}{\opnj(u)^{2/{\dcrit}}}
\ge-\left(C_3+\|\nabla_g\psi\|_{L^\infty(\cil)}^2\right)
\operatorname{Vol}_g(\exa_{L+1})^{2/n}.
\]

Taking the infimum gives
\[
Y_{\consy}(\cil)
\ge-\left({    C_3}+\|\nabla_g\psi\|_{L^\infty(\cil)}^2\right)
\operatorname{Vol}_g(\exa_{L+1})^{2/n}>-\infty.
\]

\medskip

\noindent {\it Proof of \ref{74}.} The inequality ${Y}_{\consy}(\cil)\leq {Y}_{\consy,\infty}(\cil)$ follows from the definition. We now prove that
$$
 {Y}_{\consy,\infty}(\cil)\le \Lambda^+_{n,\consy}.
$$

Fix $i\geq1$. Since $\partial\cil$ is noncompact, choose
\[
 z\in\partial\cil\setminus K_i.
\]
Let $0\ne v\in C_c^\infty(\overline{\mathbb R^n_+})$ and choose $L>0$
such that
\[
 \operatorname{supp}v\subset B_L^{n-1}\times[0,L).
\]
For the boundary normal chart $\Phi_{1,z}$ in \eqref{ch}, choose
$\epsilon>0$ so small that
\[
 \Phi_{1,z}\left(\epsilon(\operatorname{supp}v)\right)
 \subset\cil\setminus K_i,
 \qquad
 \epsilon L<r_1/2.
\]

Define
\[
 u_{\epsilon}\bigl(\Phi_{1,z}(\epsilon \xeuc)\bigr)
 :=\epsilon^{-(n-2)/2}v(\xeuc),
\]
and extend it by zero outside the chart. Then $u_{\epsilon}\in C_c^\infty(\cil\setminus K_i)\setminus\{0\}$. Moreover,
\begin{align}
 \int_{\cil}|\nabla_gu_{\epsilon}|^2\,\din v_g
 &=\int_{\mathbb R^n_+} (\nabla v)^T (\xeuc ^\prime,t)
 [g_{ij}]^{-1} (\epsilon (\xeuc ^\prime,t)) \nabla v  (\xeuc ^\prime,t)J(\epsilon (\xeuc ^\prime,t))\,\din \xeuc ^\prime \din t
 = \int_{\mathbb R^n_+}|\nabla v|^2\, \din \xeuc ^\prime \din t+O_{v}(\epsilon),\nonumber\\
 \int_{\cil}|u_{\epsilon}|^{{\dcrit}}\,\din v_g
 &=\int_{\mathbb R^n_+}|v|^{{\dcrit}}{    J}(\epsilon (\xeuc ^\prime,t))\,\din \xeuc
 =\int_{\mathbb R^n_+}|v|^{{\dcrit}}\,\din \xeuc+O_{v}(\epsilon),\nonumber\\
 \int_{\partial\cil}|u_{\epsilon}|^{{\dcritb}}\,\din s_g
 &=\int_{\mathbb R^{n-1}}|v({\xeuc ^\prime},0)|^{{\dcritb}}
 {    J^\partial}(\epsilon {\xeuc ^\prime})\,\din {\xeuc ^\prime}
 =\int_{\mathbb R^{n-1}}|v({\xeuc ^\prime},0)|^{{\dcritb}}\,\din {\xeuc ^\prime}
 +O_{v}(\epsilon),
 \label{m5}
\end{align}
where $\din v_g = J \din \xeuc ^\prime \din t$ and $\din s _g = J^\partial \din \xeuc ^\prime$.

We also have
\[
{
 \int_{\cil}u_{\epsilon}^{2}\,\din v_g
 =\epsilon^2\int_{\mathbb R^n_+}|v|^2J(\epsilon (\xeuc ^\prime,t))\,\din \xeuc ^\prime \din t,
 \qquad
 \int_{\partial\cil}u_{\epsilon}^{2}\,\din s_g
 =\epsilon\int_{\mathbb R^{n-1}}|v({\xeuc ^\prime},0)|^2
 J^\partial(\epsilon {\xeuc ^\prime})\,\din {\xeuc ^\prime}.
}
\]

Then
\begin{align}
 &\lim_{\epsilon\downarrow0}
 \frac{E_g(u_{\epsilon})}
 {\left(
 \|u_{\epsilon }\|_{L^{{\dcrit}}(\cil)}^{{\dcrit}}
 +\consy\|u_{\epsilon}\|_{L^{\dcritb}(\partial\cil)}^{{\dcrit}}
 \right)^{2/{\dcrit}}}\nonumber\\
 &\qquad=
 \frac{\displaystyle\int_{\mathbb R^n_+}|\nabla v|^2\,\din x}
 {\left(
 \|v\|_{L^{{\dcrit}}(\mathbb R^n_+)}^{\dcrit}
 +\consy\|v(\cdot,0)\|_{L^{\dcritb}(\mathbb R^{n-1})}^{\dcrit}
 \right)^{2/\dcrit}}.
 \label{m6}
\end{align}

Since $v\ne0$ is arbitrary,
\[
 {Y}_{\consy}(\cil\setminus K_i)\leq\Lambda^+_{n,\consy}
 \quad\forall i\geq1.
\]
Letting $i\to \infty$ completes the proof of \ref{74}.\end{proof}

%
%

Define
\begin{equation}
 \lambdaj
 :=\inf_{u\in W^{1,2}_{\partial}(\exa_j)\setminus\{0\}}
 Q (u),
 \label{q1}
\end{equation}
where
\begin{gather*}
Q (u)
 :=\frac{E_g(u)}{\opnj (u)^{{2}/{\dcrit}}}, \qquad \opnj (u)
 :=\|u\|_{L^{\dcrit}(\exa_j)}^{\dcrit}
 +\consy\|   u\|_{L^{\dcritb}(\exa_j\cap\partial\cil)}^{\dcrit}.
 \end{gather*}
\begin{lemma}\label{l5}
Assume that the conditions of Definition~\ref{bg} hold and that $n\ge3$ and $\consy\ge0$, and use the definitions in
\eqref{y2} and \eqref{q1}. Then
\[
 \lambdaj\searrow Y_{\consy}(\cil)
 ,
\]
where the limit is understood in $[-\infty,\infty)$.
\end{lemma}

\begin{proof}

Since $C^{\infty} _c (\exa_{j}) \subset C^{\infty} _c (\exa_{k}) $ for $j<k$, {the zero extensions also give $W^{1,2}_{\partial}(\exa_j)\subset W^{1,2}_{\partial}(\exa_k)$. Hence,} we have ${\lambdaj}   \ge
 {    \lambda_k}$. Thus 
\begin{equation}\label{29}
 {\lambdaj}   \searrow\ell\in[-\infty,+\infty).
\end{equation}

For $u\in W^{1,2}_{\partial}(\exa_j)\setminus\{0\}$,
choose $u_m\in C_c^\infty(\exa_j)$ with
$u_m\to u$ in $W^{1,2}(\exa_j)$, and
\[
 E_g(u_m)\to E_g(u),\qquad
 {\opnj}(u_m)
 \to{\opnj}(u)>0,
\]
by Propositions~\ref{pr} and \ref{26}. Hence
\[
 Y_{\consy}(\cil)\le
 \frac{E_g(u_m)}{{\opnj}(u_m)^{2/\dcrit}}
 \to    {Q}(u).
\]
Then
 \[Y_{\consy}(\cil)\le{\lambdaj}    \to \ell,
\]
by \eqref{29}.

We now prove that 
\begin{equation}\label{30}
Y_{\consy}(\cil) \ge \ell
\end{equation}

Let $0\ne\varphi\in C_c^\infty(\cil)$. There exists $j_0>0$ such that, for every $j>j_0$,
\[
 \operatorname{supp}\varphi\Subset\exa_j
\]

Thus
\[
 \ell= \lim_{j\to\infty}{\lambdaj}   
 \le {Q}(\varphi).
\]
This proves \eqref{30} and the lemma. \end{proof}

By \eqref{r2} and Lemma~\ref{l5}, there exists $j_0\in\mathbb N$ such that
\begin{equation}\label{r3}
 {\lambdaj}   <\Lambda^+_{n,\consy}
 \quad(j\ge j_0),\qquad
 \sup_{j\ge j_0}|{\lambdaj}   |<\infty.
\end{equation}


 In the proof of the following result, we shall employ the technique introduced by Br\'ezis and Nirenberg in \cite[Lemma~1.2]{BrezisNirenberg1983}; see also   \cite[Theorem~1.1 and Section~3]{ChenSun2019} and \cite[Proposition~3.2]{zbMATH06494616}.
\begin{proposition}\label{r29}
For every $j\ge j_0$, there exists
$ u_j\ge0$, $u_j\in W^{1,2}_{\partial}(\exa_j)$,
such that
\begin{equation}\label{r5}
 \left\{\begin{aligned}
 &\|u_j\|_{L^{\dcrit}(\exa_j)}^{\dcrit}
 +\consy\|   u_j\|_{L^{\dcritb}(\exa_j\cap\partial\cil)}^{\dcrit}=1,\\
 &E_g(u_j)={\lambdaj}   
 \end{aligned}\right.
\end{equation}
and $u_j$ is a weak solution of
\begin{equation}\label{24}
\left\{
\begin{aligned}
 \mathcal L_gu_j
 &={\lambdaj}   u_j^{\dcrit-1}
 &&\text{in }\exa_j\cap\operatorname{int}\cil,\\
 \mathcal B_gu_j
 &={\lambdaj}   \consy
 \|   u_j\|_{L^{\dcritb}(\exa_j\cap\partial\cil)}^{\dcrit-\dcritb}
   u_j^{\dcritb-1}
 &&\text{on }\exa_j\cap\partial\cil,
\end{aligned}
\right.
\end{equation}
where $\consy \ge 0$.

\end{proposition}

\begin{proof}

\medskip\noindent\textbf{Step 1.} We prove that, for every $\epsilon\in (0, \Lambda^+_{n,\consy})$,
\begin{equation}\label{r6}
 (\Lambda^+_{n,\consy}-\epsilon)
 {\opnj}(v)^{2/{\dcrit}}
 \le E_g(v)+C_{j,\epsilon}\|v\|_{L^2(\exa_j)}^2,
 \quad v\in W^{1,2}_{\partial}(\exa_j).
\end{equation}

For every $w\in C_c^1(\overline{\mathbb R^n_+})$, we have
\[
\begin{aligned}
&\Lambda^+_{n,\consy}
\left(
\int_{\mathbb R^n_+}|w|^{\dcrit}\,\din x
+\consy\left(\int_{\mathbb R^{n-1}}|w({\xeuc ^\prime},0)|^{\dcritb}
\,\din {\xeuc ^\prime}\right)^{{\dcrit}/{\dcritb}}
\right)^{2/{\dcrit}}\\
&\hspace{3cm}\le
\int_{\mathbb R^n_+}|\nabla w|^2\,\din x.
\end{aligned}
\]

Fix $j$.  Choose $\delta \in(0,1)$ such that
\begin{equation}\label{33}
4 \Lambda^+_{n,\consy} \delta (3-\delta) \le \epsilon 
\end{equation}

Apply
\eqref{mt}, \eqref{22}, and \eqref{dn} on boundary charts of radius
$0<r\le r_{\radpa}$, with $\radpa = (1+\delta)^{2/n}$, from \eqref{rl}.
Since $n\ge3$ and
\[
 (1+\delta)^{-1}\ge1-\delta,
 \qquad
 (1+\delta)^{2/n}\le1+\delta,
 \qquad
 (1+\delta)^{(n-1)/n}\le1+\delta,
\]
we obtain, uniformly for $z\in\partial\cil$,
\[
\begin{gathered}
 (1-\delta)I\le[g_{ij}],\ [g_{ij}]^{-1}
 \le(1+\delta)I,\\
 1-\delta\le J_z,\ J_z^\partial\le1+\delta.
\end{gathered}
\]

For $w\in C_c^1(\Phi_{1,{z}}(Q_{1,r}))$, we have
\begin{equation}\label{75}
\begin{aligned}
&{\opnj}(w)^{2/{\dcrit}}\\
&\le(1+\delta)^{2/{\dcritb}}
\left(
 \int_{\mathbb R^n_+}|w\circ\Phi_{{1},z}|^{\dcrit}\,\din x
 +\consy\left(\int_{\mathbb R^{n-1}}
 |(w\circ\Phi_{{1},z})({\xeuc ^\prime},0)|^{\dcritb}\,\din {\xeuc ^\prime}
 \right)^{{\dcrit}/{\dcritb}}
\right)^{2/{\dcrit}}\\
&\le\frac{(1+\delta)^{2/{\dcritb}}}
{\Lambda^+_{n,\consy}}
\int_{\mathbb R^n_+}|\nabla(w\circ\Phi_{{1},z})|^2\,\din x\\
&\le\frac{(1+\delta)^{2/{\dcritb}}}
{\Lambda^+_{n,\consy}(1-\delta)^2}
\int_\cil|\nabla_gw|^2\,\din v_g.
\end{aligned}
\end{equation}

{Since $2/\dcritb=(n-2)/(n-1)<1$, \eqref{33} gives
\[
 \begin{aligned}
 \Lambda^+_{n,\consy}\frac{(1-\delta)^2}{(1+\delta)^{2/\dcritb}}
 &\ge\Lambda^+_{n,\consy}\frac{(1-\delta)^2}{1+\delta}
 =\Lambda^+_{n,\consy}\left(1-\frac{\delta(3-\delta)}{1+\delta}\right)\\
 &\ge\Lambda^+_{n,\consy}-\epsilon/4.
 \end{aligned}
\]
}

Then, by \eqref{75}, we have
\begin{equation}\label{31}
\left(\Lambda^+_{n,\consy} -\epsilon /4\right){\opnj}(w)^{2/{\dcrit}}
\le\int_\cil|\nabla_gw|^2\,\din v_g.
\end{equation}

{For an interior chart of radius $0<a\le r_{z,\radpa}$ and $w\in C_c^1(\Phi_{0,z}(Q_{0,a}))$, the boundary term vanishes. By \eqref{46},
\[
 \begin{aligned}
 {\opnj}(w)^{2/\dcrit}
 &\le\consob\radpa^n\int_\cil|\nabla_gw|^2\,\din v_g\\
 &\le\frac{(1+\delta)^{2/\dcritb}}{\Lambda^+_{n,\consy}(1-\delta)^2}
 \int_\cil|\nabla_gw|^2\,\din v_g,
 \end{aligned}
\]
where $\consob\le(\Lambda^+_{n,\consy})^{-1}$, $\radpa=(1+\delta)^{2/n}$ and $(1-\delta^2)^2\le1$. Hence \eqref{31} also holds for interior cutoffs.}

Let $v\in C^{1} _c(\exa_j )$, and let $\{\chi _{\alpha}\}_{\alpha \in \mathtt{J}} :=\{\eta _{0,j}\}_{j\in \mathtt{J}_0}\cup \{\eta _{1,j}\}_{j\in \mathtt{J}_1}$ be the partition given in Lemma~\ref{pt} with $q=2$. {Choose $r:=\min\{r_{\radpa},s_0/4\}$ and $s:=r/32$ with the above $\radpa$ and the geometric radius $s_0$ from Lemma~\ref{pt}; \eqref{ms} then holds, and both families satisfy \eqref{31}.} Hence
\[
\begin{aligned}
\|v\|_{L^{\dcrit}(\exa_j)}^2
= &\left\|\sum_\alpha|\chi_\alpha  v|^2
\right\|_{L^{{\dcrit}/2}(\exa_j)}\\
\le & \sum_\alpha\|\chi_\alpha  v\|_{L^{\dcrit}(\exa_j)}^2,
\end{aligned}
\]
and
\[
\begin{aligned}
\|   v\|_{L^{\dcritb}(\exa_j\cap\partial\cil)}^2
=&\left\|\sum_\alpha
|   (\chi_\alpha  v)|^2
\right\|_{L^{{\dcritb}/2}(\exa_j\cap\partial\cil)}\\
\le&\sum_\alpha
\|   (\chi_\alpha  v)\|_{L^{\dcritb}(\exa_j\cap\partial\cil)}^2.
\end{aligned}
\]

Hence the inequality $[(\sum a_i)^r+(\sum b_i)^r]^{1 / r} \leq \sum  (a_i^r+b_i^r)^{1 / r}$, for $a_i,b_i\ge 0$ and $r>1$, gives
\begin{equation}\label{32}
\begin{aligned}
&{\opnj}(v)^{2/{\dcrit}}\\
&=\left[
\bigl(\|v\|_{L^{\dcrit}(\exa_j)}^2\bigr)^{{\dcrit}/2}
+\bigl(\consy^{2/{\dcrit}}
\|   v\|_{L^{\dcritb}(\exa_j\cap\partial\cil)}^2
\bigr)^{{\dcrit}/2}
\right]^{2/{\dcrit}}\\
&\le\left[
\left(\sum_\alpha\|\chi_\alpha  v\|_{L^{\dcrit}(\exa_j)}^2\right)^{{\dcrit}/2}
+\left(\sum_\alpha\consy^{2/{\dcrit}}
\|   (\chi_\alpha   v)\|_{L^{\dcritb}(\exa_j\cap\partial\cil)}^2
\right)^{{\dcrit}/2}
\right]^{2/{\dcrit}}\\
&\le\sum_\alpha
\left(
\|\chi_\alpha  v\|_{L^{\dcrit}(\exa_j)}^{\dcrit}
+\consy\|   (\chi_\alpha  v)\|_{L^{\dcritb}(\exa_j\cap\partial\cil)}^{\dcrit}
\right)^{2/{\dcrit}}\\
&=\sum_\alpha{\opnj}
(\chi_\alpha  v)^{2/{\dcrit}}.
\end{aligned}
\end{equation}

By \eqref{31} and \eqref{32},
\[
(\Lambda^+_{n,\consy}-\epsilon/4)
{\opnj}(v)^{2/{\dcrit}} \le\sum_\alpha
\int_\cil|\nabla_g(\chi_\alpha v)|^2\,\din v_g.
\]
Using $\sum_\alpha\chi_\alpha\nabla_g\chi_\alpha=0$  and \eqref{pw},
\[
\begin{aligned}
\sum_\alpha|\nabla_g(\chi_\alpha  v)|^2
&=|\nabla_g  v|^2
+  v^2\sum_\alpha|\nabla_g\chi_\alpha|^2,\\
\sum_\alpha\int_\cil
|\nabla_g(\chi_\alpha v)|^2\,\din v_g
&\le\int_{\exa_j}|\nabla_gv|^2\,\din v_g
+{\bf c}_2 (\epsilon, n, g )\|v\|_{L^2(\exa_j)}^2.
\end{aligned}
\]

Consequently,
\begin{equation}\label{76}
(\Lambda^+_{n,\consy}-\epsilon/4)
{\opnj}(v)^{2/{\dcrit}}
\le\int_{\exa_j}|\nabla_gv|^2\,\din v_g
+{\bf c}_2 (\epsilon, n, g ) \|v\|_{L^2(\exa_j)}^2.
\end{equation}

 Set $C_2 := \frac{n-2}{4(n-1)} \sup _{B_j}R_g^{-}$ and $C_3:= \frac{n-2}{2(n-1)} \sup _{B_j{\cap\partial\cil}}H_g^{-}$. Applying \eqref{n3} with $\epsilon/[4\Lambda^+_{n,\consy}(1+C_3)]$ gives
\begin{equation}\label{77}
\begin{aligned}
E_g(v)
&\ge\int_{\exa_j}|\nabla_gv|^2\,\din v_g
-C_2\|v\|_{L^2(\exa_j)}^2
-C_3\|   v\|_{L^2(\exa_j\cap\partial\cil)}^2\\
&\ge\left(1-\frac{\epsilon C_3}
{4\Lambda^+_{n,\consy}(1+C_3)}\right)
\int_{\exa_j}|\nabla_gv|^2\,\din v_g
-(C_2 + C_4(\epsilon,g))\|v\|_{L^2(\exa_j)}^2\\
&\ge\left(1-\frac{\epsilon}{4\Lambda^+_{n,\consy}}\right)
\int_{\exa_j}|\nabla_gv|^2\,\din v_g
-C_5\|v\|_{L^2(\exa_j)}^2
\end{aligned}
\end{equation}
where $C_5:= C_2 + C_4$.

Combining \eqref{76} and \eqref{77}, we obtain
\[
\begin{aligned}
&\left(1-\frac{\epsilon}{4\Lambda^+_{n,\consy}}\right)
(\Lambda^+_{n,\consy}-\epsilon/4)
{\opnj}(v)^{2/{\dcrit}}\\
&\qquad\le E_g(v)+\left(C_5{+\left(1-\frac{\epsilon}{4\Lambda^+_{n,\consy}}\right){\bf c}_2(\epsilon,n,g)}\right)\|v\|_{L^2(\exa_j)}^2.
\end{aligned}
\]
Since
\[
\begin{aligned}
\left(1-\frac{\epsilon}{4\Lambda^+_{n,\consy}}\right)
(\Lambda^+_{n,\consy}-\epsilon/4)
&=\Lambda^+_{n,\consy}-\frac{\epsilon}{2}
+\frac{\epsilon^2}{16\Lambda^+_{n,\consy}}\\
&\ge\Lambda^+_{n,\consy}-\epsilon,
\end{aligned}
\]
we obtain \eqref{r6} for $v\in C_c^1(\exa_j)$.

\medskip

\noindent\textbf{Step 2.}
Choose $w_k \in W^{1,2} _\partial (\exa_j) $ with $w_k\ge0$ such that
\[
 {\opnj}(w_k)
 =\|w_k\|_{L^{\dcrit}(\exa_j)}^{\dcrit}
 +\consy\|  w_k\|_{L^{\dcritb}(\exa_j\cap\partial\cil)}^{\dcrit} =1,
 \qquad E_g(w_k)\to   \lambdaj.
\]

Then,
$$
 \|w_k\|_{L^2(\exa _j)}^2
 \le\vol_g (\exa_j)^{1-2/{\dcrit}}\|w_k\|_{L ^{\dcrit}(B_j)} ^2
 \le\vol _g (\exa_j)^{1-2/{\dcrit}},
 $$

On the other hand, for all $w\in W^{1,2}_\partial (\exa _j)$, we have
\begin{equation*}
 \|\nabla_gw\|_{L^2(\exa _j)}^2
 \le E_g(w) + \sup _{\exa _j}R_g^- \|w\|_{L^2(\exa _j)}^2 + \sup _{\exa _j\cap \partial \cil} H_g^-  \| w\|_{L^2(\exa _j\cap \partial\cil)}^2 .
 \end{equation*}
Then, by \eqref{n3} with $\epsilon_0 =2^{-1}(1+\sup _{\exa _j\cap \partial \cil} H_g^-)^{-1}$,
\[
 \frac12\|\nabla_gw\|_{L^2(\exa _j)}^2
 \le E_g(w)+\left(\sup_{\exa_j}R_g^-+C_6\sup_{\exa_j\cap\partial\cil}H_g^-\right)\|w\|_{L^2(\exa _j)}^2.
\]

Thus
\begin{equation} \label{35}
\|\nabla_gw\|_{L^2(\exa _j)}^2
 \le {2E_g(w)+C_7\|w\|_{L^2(\exa _j)}^2}
\end{equation}
where $C_7:=  2(1+\sup _{\exa _j} R_g^-+ C_6\sup _{\exa _j\cap \partial \cil} H_g^-)$ and ${C_6=C_6(\epsilon_0,g)}>0$.

Rellich compactness and compactness of the trace yield, after passage to a subsequence,
\begin{equation}\label{r7}
 \begin{aligned}
 w_k&\rightharpoonup u_j
 &&\text{in }W^{1,2}_{\partial}(\exa_j),\\
 w_k&\to    u_j
 &&\text{in }L^2(\exa_j)\text{ and a.e.},\\
    w_k&\to      u_j
 &&\text{in }L^2(\exa_j\cap\partial\cil)\text{ and a.e.}.
 \end{aligned}
\end{equation}

By Fatou's lemma,
\begin{equation}\label{c1}
{\opnj}(u_j)\le1.
\end{equation}

Set  $v_k:=w_k-u_j$. Since 
$$
 \|\nabla_gw_k\|_{L^2(\cil)}^2
 =\|\nabla_gu_j\|_{L^2(\cil)}^2+\|\nabla_gv_k\|_{L^2(\cil)}^2
 +2\int _{\cil }g \left( {\nabla_g}u_j,\nabla_gv_k\right) \, \din v_g
$$
we obtain
\begin{equation}\label{36}
 \lambdaj = E_g(u_j)+\lim _{k\to \infty} E_g(v_k) .
\end{equation}

By Theorem~1 of Br\'ezis and Lieb \cite{BrezisLieb1983}, we obtain
\begin{equation}\label{c2}
\begin{aligned}
\|w_k\|_{L^{\dcrit}(\exa_j)}^{\dcrit}
&=\|u_j\|_{L^{\dcrit}(\exa_j)}^{\dcrit}
 +\|v_k\|_{L^{\dcrit}(\exa_j)}^{\dcrit}+o(1),\\
\|w_k\|_{L^{\dcritb}(\exa_j\cap\partial\cil)}^{\dcritb}
&=\|u_j\|_{L^{\dcritb}(\exa_j\cap\partial\cil)}^{\dcritb}
 +\|v_k\|_{L^{\dcritb}(\exa_j\cap\partial\cil)}^{\dcritb}+o(1).
\end{aligned}
\end{equation}

By \eqref{c2} and the inequality 
$$
[(a_1+a_2)+c(b_1+b_2)^{r / s}]^{1 / r} \leq(a_1+c b_1^{r / s})^{1 / r}+(a_2+c b_2^{r / s})^{1 / r}
, \quad a_i,b_i,c\ge 0, r,s>1,
$$
 we obtain, for $r=\dcrit/2$ and $s=\dcritb/2$,
\begin{equation}\label{c3}
\begin{aligned}
1
&={\opnj}(w_k)^{2/\dcrit}\\
&=\Bigl[\|u_j\|_{L^{\dcrit}(\exa_j)}^{\dcrit}
+\|v_k\|_{L^{\dcrit}(\exa_j)}^{\dcrit}
+\consy\Bigl(\|u_j\|_{L^{\dcritb}(\exa_j\cap\partial\cil)}^{\dcritb}
+\|v_k\|_{L^{\dcritb}(\exa_j\cap\partial\cil)}^{\dcritb}\Bigr)^{\dcrit/\dcritb}
+o(1)\Bigr]^{2/\dcrit}\\
&\le{\opnj}(u_j)^{2/\dcrit}
+{\opnj}(v_k)^{2/\dcrit}+o(1).
\end{aligned}
\end{equation}

Passing to a subsequence, assume
\[
{\opnj}(v_k)^{2/\dcrit}\to    d\ge0.
\]
Consequently, \eqref{c3} yields
\begin{equation}\label{c4}
1\le{\opnj}(u_j)^{2/\dcrit}+d.
\end{equation}

For $0<\epsilon<\Lambda^+_{n,\consy}$, applying \eqref{r6}
to $v_k$ and using \eqref{r7} gives
\[
(\Lambda^+_{n,\consy}-\epsilon)d
\le\liminf_{k\to\infty}E_g(v_k).
\]
Letting $\epsilon\downarrow0$ and using \eqref{36} together with $\lambdaj<\Lambda^+_{n,\consy}$ from \eqref{r3}, we obtain
\begin{equation}\label{c5}
\lambdaj\ge E_g(u_j)+\Lambda^+_{n,\consy}d
\ge\lambdaj {\opnj}(u_j)^{2/\dcrit}
+\Lambda^+_{n,\consy}d.
\end{equation}
If $\lambdaj\ge0$, \eqref{c4} and \eqref{c5} imply
\[
\lambdaj\ge\lambdaj +(\Lambda^+_{n,\consy}-\lambdaj)d.
\]
If $\lambdaj<0$, \eqref{c1} and \eqref{c5} imply
\[
\lambdaj\ge\lambdaj+\Lambda^+_{n,\consy}d.
\]
In both cases, $\Lambda^+_{n,\consy}>\lambdaj$ and $\Lambda^+_{n,\consy}>0$ imply $d=0$.

Therefore
\begin{equation}\label{c6}
\|v_k\|_{L^{\dcrit}(\exa_j)}\to0,\qquad
\consy \|v_k\|_{L^{\dcritb}(\exa_j\cap\partial\cil)} ^{\dcrit}\to0.
\end{equation}

\medskip

\noindent\textbf{Step 3.}
By \eqref{c2}, \eqref{c6}, and the normalization of $w_k$,
\[
{\opnj}(u_j)=1,
\qquad u_j\ge0,\quad u_j\not\equiv0.
\]
By \eqref{r6}, \eqref{r7}, and \eqref{c6},
\[
0\le\liminf_{k\to\infty}E_g(v_k).
\]
Therefore, by \eqref{36},
\[
\lambdaj\le E_g(u_j)
=\lambdaj-\lim_{k\to\infty}E_g(v_k)
\le\lambdaj,
\]
so $E_g(u_j)=\lambdaj$.

Since $u_j$ minimizes \eqref{q1},
for every $\varphi\in W^{1,2}_{\partial}(\exa_j)$,
\[
\begin{aligned}
&\int_{\exa_j}\left(\langle\nabla_gu_j,\nabla_g\varphi\rangle_g
+\frac{n-2}{4(n-1)}R_gu_j\varphi\right)\,\din v_g
+\frac{n-2}{2(n-1)}\int_{\exa_j\cap\partial\cil}
H_gu_j\varphi\,\din s_g\\
&\qquad=\lambdaj\int_{\exa_j}u_j^{\dcrit-1}\varphi\,\din v_g
+\lambdaj\consy\|u_j\|_{L^{\dcritb}(\exa_j\cap\partial\cil)}^{\dcrit-\dcritb}
\int_{\exa_j\cap\partial\cil}u_j^{\dcritb-1}\varphi\,\din s_g.
\end{aligned}
\]
This is the weak formulation of \eqref{24}.\end{proof}

\section{Uniform local bounds} \label{71}

Let $u_j\in W^{1,2}_\partial (\exa _j)$ be the solutions obtained in Proposition~\ref{r29}.

Set
\[
 \gamma_j:=\consy\|   u_j\|_{L^{\dcritb}(\exa_j\cap\partial\cil)}^{{\dcrit}-{\dcritb}}.
\]
Then
\begin{equation}\label{r10}
 \|u_j\|_{L^{\dcrit}(\exa_j)}^{\dcrit}
 +\consy\|   u_j\|_{L^{\dcritb}(\exa_j\cap\partial\cil)}^{\dcrit}=1,
 \qquad 0\le\gamma_j\le\consy^{{\dcritb}/{\dcrit}}.
\end{equation}


 To prove the following proposition, we follow the strategy of Wei \cite[Lemma~2.5 and Theorem~4.1]{Wei2019}; see also \cite{MarinoWinkert2019}.
\begin{proposition}\label{r12}
Under \eqref{r2} and the assumptions of Definition~\ref{bg}, with $\consy \ge 0$, we have
\begin{equation}\label{r14}
\displaystyle
 \sup_{j\ge j_0}\|u_j\|_{L^\infty(\exa_{j-1})}<\infty.
\end{equation}

\end{proposition}

\begin{proof}
\medskip\noindent\textbf{Step 1.} In Lemma~\ref{pt}, choose $\radpa>1$ close to $1$ with $r_{\radpa}<1/4$. This gives the partition of unity 
$$
\begin{aligned}
\eta _{h,i} & : \Phi _{{h},z_{h,i}} ( Q_{h,r_h})\to [0,1], \qquad  \spt {\eta_{h,i}} = \Phi _{h,z_{h,i}} ({\overline{Q_{h,3r_h/4}}}) ,\quad i\in \mathtt{J}_h,  \ h=0,1
 \end{aligned} 
$$ 
where $r_1,r_0\in(0,1/4)$ are fixed. Moreover,
\begin{equation}\label{50}
\cil = \bigcup _{i\in \mathtt{J}_1} \Phi _{1,z_{1,i}} (Q_{1,r_1 /2})  \cup \bigcup _{i\in \mathtt{J}_0} \Phi _{0,z_{0,i}} (Q_{{    0},r_0 /2}) 
\end{equation}
 By \eqref{22} and \eqref{45},
 \begin{equation}\label{51}
\vol _g( \Phi _{h,z_{h,i}} {(Q_{h,r_h})} ) \le \radpa ^{n/2}  \operatorname{Vol} (Q_{h,r_h})
 \end{equation}

If ${\Phi _{h,z_{h,i}}(Q_{h,r_h})}\cap \exa_{j-1}{\ne\varnothing}$, then
$$
\Phi _{h,z_{h,i}} (Q_{h,r_h})\subset  \exa_{j}
$$

{Indeed, for $x=\Phi_{1,z}(\xeuc',t)$ the boundary radial geodesic followed by the normal segment gives
\[
 d_g(z,x)\le |\xeuc'|+t<2r_1,
 \qquad \operatorname{diam}_g\Phi_{1,z}(Q_{1,r_1})\le4r_1<1.
\]
For an interior chart,
\[
 d_g(z,\Phi_{0,z}(\xeuc))=|\xeuc|<r_0,
 \qquad \operatorname{diam}_g\Phi_{0,z}(Q_{0,r_0})\le2r_0<1.
\]
Thus, if $x$ belongs to the outer chart and to $\exa_{j-1}$, every $y$ in that chart satisfies
\[
 d_g(\centro,y)\le d_g(\centro,x)+d_g(x,y)<(j-1)+1=j.
\]
All charts used below are selected by this nonempty-intersection condition.}

For simplicity, write
$$
V:= \Phi _{h,z_{h,i}} (Q_{h,r_h})\qquad W:= \Phi_{h,z_{h,i}}(Q_{h,3r_h/4}) .
$$

By \eqref{31}, for $\epsilon\in(0,\Lambda^+_{n,\consy})$ {fixed as in Step~3, choose $\radpa>1$ sufficiently close to $1$ and the chart radii sufficiently small. The boundary and interior versions of \eqref{31} then give}
\begin{equation}\label{r15}
 (\Lambda^+_{n,\consy}-\epsilon)
 \left(\|f\|_{L^{\dcrit}(V   )}^{\dcrit}
 +\consy\|   f\|_{L^{\dcritb}(V   \cap\partial\cil)}^{\dcrit}\right)^{2/{\dcrit}}
 \le\int_{V   }|\nabla_gf|^2\,\din v_g,
 \quad f\in W^{1,2}(V   ),\quad \operatorname{supp}f\Subset V   
\end{equation}
Moreover, by H\"older's inequality and $\gamma _j = \consy \|u_j\|^{\dcrit - \dcritb}_{L^{\dcritb} (\exa _j \cap {\partial}\cil)}$, we have
\begin{equation}\label{r16}
 \begin{aligned}
 &\int_{V   }u_j^{{\dcrit}-2}f^2\,\din v_g
 +\gamma_j\int_{V   \cap\partial\cil}
    u_j^{{\dcritb}-2}   f^2\,\din s_g\\
 &\le\|u_j\|_{L^{\dcrit}(\exa _j)}^{{\dcrit}-2}\|f\|_{L^{\dcrit}(V   )}^2
 +\consy\|   u_j\|_{L^{\dcritb}(\exa _j{\cap\partial\cil})}^{{\dcrit}-2}
 \|   f\|_{L^{\dcritb} (V   {\cap\partial\cil})}^2\\
 &\le
 \left(\|u_j\|_{L^{\dcrit}  (\exa _j)}^{\dcrit}
 +\consy\|   u_j\|_{L^{\dcritb}(\exa _j{\cap\partial\cil})}^{\dcrit}\right)^{({\dcrit}-2)/{\dcrit}}
 \left(\|f\|_{L^{\dcrit}(V   )}^{\dcrit}
 +\consy\|   f\|_{L^{\dcritb} (V   {\cap\partial\cil})}^{\dcrit}\right)^{2/{\dcrit}}\\
 &=\left(\|f\|_{L^{\dcrit}(V   )}^{\dcrit}
 +\consy\|   f\|_{L^{\dcritb} (V   {\cap\partial\cil})}^{\dcrit}\right)^{2/{\dcrit}}.
 \end{aligned}
\end{equation}

\medskip

\noindent\textbf{Step 2.} For $\beta\ge1$, define
\[
 \theta_\beta:=\frac{2\beta-1}{\beta^2}.
\]

Choose $\eta\in C_c^\infty(V)$,
with
\[
 0\le\eta\le1,\qquad
 \operatorname{supp}\eta\Subset V,\qquad
 |\nabla_g\eta|\le C/r_h.
 \]
 Set
\[
 u_{j,L}:=\min\{u_j,L\},\quad
 w_L:=u_j u_{j,L}^{\beta-1},\quad
 \varphi_L:=\eta^2u_j u_{j,L}^{2\beta-2}.
\]

We will show that, if $\theta _\beta \ge \delta >0$, then
 \begin{equation}\label{r18}
 \begin{aligned}
 &(\theta_\beta-\delta)
 \int_{V   }|\nabla_g(\eta w_L)|^2\,\din v_g\\
 &\le{\lambdaj ^+}   \left[
 \int_{V   }u_j^{{\dcrit}-2}(\eta w_L)^2\,\din v_g
 +\gamma_j\int_{V   \cap\partial\cil}
 (   u_j)^{{\dcritb}-2}(   \eta w_L)^2\,\din s_g
 \right]\\
 &\quad+C_{\beta,\delta}\int_{V   }
 (\eta^2+|\nabla_g\eta|^2)w_L^2\,\din v_g  
 \end{aligned}
\end{equation}

For $\varphi_L=\eta^2u_ju_{j,L}^{2\beta-2}$, \eqref{24} gives
\begin{align}
&\int_{V   }\langle\nabla_gu_j,\nabla_g\varphi_L\rangle_g\,\din v_g \nonumber\\
&=\lambdaj\left[
\int_{V   }u_j^{{\dcrit}-2}(\eta w_L)^2\,\din v_g
+\gamma_j\int_{V   \cap\partial\cil}
 u_j^{{\dcritb}-2}(\eta w_L)^2\,\din s_g\right] \nonumber \\
&\quad-\frac{n-2}{4(n-1)}\int_{V   }R_g(\eta w_L)^2\,\din v_g
-\frac{n-2}{2(n-1)}\int_{V   \cap\partial\cil}
 H_g(\eta w_L)^2\,\din s_g \nonumber\\
 &\le{\lambdaj ^+}   \left[
\int_{V   }u_j^{{\dcrit}-2}(\eta w_L)^2\,\din v_g
+\gamma_j\int_{V   \cap\partial\cil}
 u_j^{{\dcritb}-2}(\eta w_L)^2\,\din s_g\right] \nonumber\\
&\quad+ C_1
\int_{V   }\eta^2w_L^2\,\din v_g 
 + C_2
\int_{V   \cap\partial\cil}(\eta w_L)^2\,\din s_g,\label{40}
\end{align}
where $C_1 := \frac{n-2}{4(n-1)}\|R_g^-\|_{L^\infty(\cil)}$ and  $C_2:= \frac{n-2}{2(n-1)}\|H_g^-\|_{L^\infty(\partial\cil)}$.

Almost everywhere, we have
\begin{equation}\label{62}
\nabla_gw_L
=\left\{\begin{aligned}
&\beta u_j^{\beta-1}\nabla_gu_j & &u_j\le L,\\
&L^{\beta-1}\nabla_gu_j &&u_j>L,
\end{aligned}\right.
\end{equation}
so
\begin{equation*}
u_{j,L}^{\beta-1}|\nabla_gu_j|\le|\nabla_gw_L| .
\end{equation*}

We also have
\begin{equation*}
\langle\nabla_gu_j,\nabla_g(u_ju_{j,L}^{2\beta-2})\rangle_g
=\left\{\begin{aligned}
&\dfrac{2\beta-1}{\beta^2}|\nabla_gw_L|^2,&&u_j\le L,\\
&|\nabla_gw_L|^2,&&u_j>L .
\end{aligned}\right.
\end{equation*}

Then
\[
\begin{aligned}
\langle\nabla_gu_j,\nabla_g\varphi_L\rangle_g
= & \left\langle  \nabla _g u_j , \eta^2\nabla_g(u_ju_{j,L}^{2\beta-2})
+2\eta u_ju_{j,L}^{2\beta-2}\nabla_g\eta \right\rangle_g\\
\ge&\theta_\beta\eta^2|\nabla_gw_L|^2
-2\eta w_L|\nabla_gw_L||\nabla_g\eta|,\\
\ge& \theta_\beta\eta^2|\nabla_gw_L|^2 -\frac{\delta}{4}\eta^2|\nabla_gw_L|^2
{-}\frac4\delta w_L^2|\nabla_g\eta|^2.
\end{aligned}
\]
Consequently,
\begin{equation}\label{37}
\begin{aligned}
&\left(\theta_\beta-\frac\delta4\right)
\int_{V   }\eta^2|\nabla_gw_L|^2\,\din v_g\\
&\quad\le\int_{V   }\langle\nabla_gu_j,\nabla_g\varphi_L\rangle_g\,\din v_g
+\frac4\delta\int_{V   }w_L^2|\nabla_g\eta|^2\,\din v_g.
\end{aligned}
\end{equation}

On the other hand, from $\nabla_g(\eta w_L)=\eta\nabla_gw_L+w_L\nabla_g\eta,$  we obtain
\begin{equation}\label{39}
|\nabla_g(\eta w_L)|^2
\le\left(1+\frac\delta4\right)\eta^2|\nabla_gw_L|^2
+\left(1+\frac4\delta\right)w_L^2|\nabla_g\eta|^2,
\end{equation}

From \eqref{37} and \eqref{39},
\begin{equation}\label{41}
\begin{aligned}
&\frac{\theta_\beta-\delta/4}{1+\delta/4}
\int_{V   }|\nabla_g(\eta w_L)|^2\,\din v_g\\
&\quad\le\int_{V   }\langle\nabla_gu_j,\nabla_g\varphi_L\rangle_g\,\din v_g
+\frac4\delta\left(1+\theta_\beta-\frac\delta4\right)
\int_{V   }w_L^2|\nabla_g\eta|^2\,\din v_g.
\end{aligned}
\end{equation}

To complete the proof of \eqref{r18}, choose $\epsilon_0=\frac{\delta}{4\left(1+C_2\right)}>0$ in \eqref{n3}. Then
\begin{equation}\label{43}
\begin{aligned}
&\int_{V   \cap\partial\cil}(\eta w_L)^2\,\din s_g\\
&\quad\le\frac{\delta}{4\left(1+C_2\right)} \int_{V   }|\nabla_g(\eta w_L)|^2\,\din v_g
+C_3(\delta, C_2, g)
\int_{V   }\eta^2w_L^2\,\din v_g .
\end{aligned}
\end{equation}

Since
$$
\frac{\theta_\beta-\delta/4}{1+\delta/4}
-\left(\theta_\beta-\frac\delta2\right)=\frac{\frac\delta4(1-\theta_\beta)+\frac{\delta^2}{8}}
{1+\delta/4}\ge0,
$$
we have
\begin{equation}\label{42}
\frac{\theta_\beta-\delta/4}{1+\delta/4}
-\frac{\delta C_2}{4\left(1+C_2\right)}\ge\theta_\beta-\frac\delta2-\frac\delta4
>\theta_\beta-\delta\ge0.
\end{equation}

Thus \eqref{r18} follows from \eqref{41}, \eqref{40}, \eqref{43}, and \eqref{42}. Furthermore, $C_{\beta,\delta}=\max\{
\frac4\delta(1+\theta_\beta-\frac\delta4)+C_1
+C_2C_3\}$.

\medskip

\noindent {\bf Step 3.}  By \eqref{r2} and \eqref{r3}, fix $\epsilon  $, $\delta >0$, and $\beta >1$ so that $\beta < \dcrit /2$ and
\begin{equation}\label{r17}
 (\theta _\beta-2{\delta})(\Lambda^+_{n,\consy}-\epsilon)
 >\sup_{j\ge j_0}{\lambdaj ^+}   .
\end{equation}

 We will show that
\begin{equation}\label{r19}\left(\|\eta w_L\|_{L^{\dcrit}(V   )}^{\dcrit}
 +\consy\|   (\eta w_L)\|_{L^{\dcritb}(V   \cap\partial\cil)}^{\dcrit}\right)^{2/{\dcrit}}\le C_3,
\end{equation}
where $C_3>0$ is independent of $j,h,i,L$, and $u_j$.


By \eqref{r15}, \eqref{r16}, and \eqref{r18},
\[
\begin{aligned}
&(\theta_\beta-\delta)(\Lambda^+_{n,\consy}-\epsilon)
\left(\|\eta w_L\|_{L^{\dcrit}(V   )}^{\dcrit}
+\consy\|\eta w_L \|_{L^{\dcritb}(V   \cap\partial\cil)}^{\dcrit}\right)^{2/\dcrit}\\
&\le(\theta_\beta-\delta)
\int_{V   }|\nabla_g (\eta w_L)|^2\,\din v_g\\
&\le{\lambdaj ^+}   \left[
\int_{V   }u_j^{\dcrit-2}(\eta w_L)^2\,\din v_g
+\gamma_j\int_{V   \cap\partial\cil}
 u_j^{\dcritb-2}(\eta w_L)^2\,\din s_g\right]\\
&\qquad+C_{\beta,\delta}\int_{V   }
(\eta^2+|\nabla_g\eta|^2)w_L^2\,\din v_g\\
&\le{\lambdaj ^+}   
\left(\|\eta w_L\|_{L^{\dcrit}(V   )}^{\dcrit}
+\consy\|\eta w_L\|_{L^{\dcritb}(V   \cap\partial\cil)}^{\dcrit}\right)^{2/\dcrit}\\
&\qquad+C_{\beta,\delta}\int_{V   }
(\eta^2+|\nabla_g\eta|^2)w_L^2\,\din v_g.
\end{aligned}
\]
Furthermore, by \eqref{r17},
\[
\begin{aligned}
&(\theta_\beta-\delta)(\Lambda^+_{n,\consy}-\epsilon)
-{\lambdaj ^+}   \\
&\quad=\big[(\theta_\beta-2\delta)(\Lambda^+_{n,\consy}-\epsilon)
-{\lambdaj ^+}   \big]
+\delta(\Lambda^+_{n,\consy}-\epsilon)\\
&\quad\ge\delta(\Lambda^+_{n,\consy}-\epsilon)>0,
\qquad j\ge j_0.
\end{aligned}
\]
Hence
\begin{equation}\label{48}
\begin{aligned}
&\left(\|\eta w_L\|_{L^{\dcrit}(V   )}^{\dcrit}
+\consy\|\eta w_L\|_{L^{\dcritb}(V   \cap\partial\cil)}^{\dcrit}\right)^{2/\dcrit}\\
&\quad\le\frac{C_{\beta,\delta}}
{\delta(\Lambda^+_{n,\consy}-\epsilon)}
\int_{V   }(\eta^2+|\nabla_g\eta|^2)w_L^2\,\din v_g\\
&\quad\le\frac{C_{\beta,\delta}}
{\delta(\Lambda^+_{n,\consy}-\epsilon)}
\left(\|\eta\|_{L^\infty(V   )}^2
+\|\nabla_g\eta\|_{L^\infty(V   )}^2\right)
\int_{V   }w_L^2\,\din v_g.
\end{aligned}
\end{equation}

Since $w_L=u_j\min\{u_j,L\}^{\beta-1}
\le u_j^\beta$, we have
\begin{equation}\label{49}
\begin{aligned}
\int_{V   }w_L^{2}\,\din v_g
\le & \int_{V   }u_j^{2\beta}\,\din v_g
\le\left(\int_{V   }u_j^{\dcrit}\,\din v_g\right)^{2\beta/\dcrit}
\left(\int_{V   }1\,\din v_g\right)^{1-2\beta/\dcrit}\\
=&\operatorname{Vol}_g(V   )^{1-2\beta/\dcrit}
\|u_j\|_{L^{\dcrit}(V   )}^{2\beta}\\
\le&\operatorname{Vol}_g(V   )^{1-2\beta/\dcrit}
\|u_j\|_{L^{\dcrit}(\exa_j)}^{2\beta}
\le\operatorname{Vol}_g(V   )^{1-2\beta/\dcrit},
\end{aligned}
\end{equation}
where the last step follows from \eqref{r5}.

By \eqref{48} and \eqref{49},
\[
\begin{aligned}
&\left(\|\eta w_L\|_{L^{\dcrit}(V   )}^{\dcrit}
+\consy\|\eta w_L\|_{L^{\dcritb}(V   \cap\partial\cil)}^{\dcrit}\right)^{2/\dcrit}\\
&\quad\le\frac{C_{\beta,\delta}}
{\delta(\Lambda^+_{n,\consy}-\epsilon)}
\left(\|\eta\|_{L^\infty(V   )}^2
+\|\nabla_g\eta\|_{L^\infty(V   )}^2\right)
\operatorname{Vol}_g(V   )^{1-2\beta/\dcrit}
\end{aligned}
\]
By \eqref{50} and \eqref{51}, the right-hand side is bounded by a constant independent of $\exa_j$ and $u_j$. This proves \eqref{r19}.

\medskip

\noindent {\bf Step 4.} For $f\in W^{1,2}(\cil)$ with $\spt f\Subset W$, $m=\frac{\beta+1}{\beta-1}$, and $0<h<1$, we will show that 
\begin{equation}\label{r22}
\begin{aligned}
&\int_Wu_j^{\dcrit-2}f^2\,\din v_g
 +\gamma_j\int_{W\cap\partial\cil}u_j^{\dcritb-2}f^2\,\din s_g\\
&\qquad\le h\int_\cil|\nabla_g f|^2\,\din v_g
 +C_4(1+h^{-m})\int_\cil f^2\,\din v_g, 
\end{aligned}
\end{equation}
where $C_4>0$ is independent of $u_j,h,f$ {and of the chart; it may depend on the fixed $\beta$}.

Recall that
\[
 V=\Phi_{h,z_{h,i}}(Q_{h,r_h}),\qquad
 W=\Phi_{h,z_{h,i}}(Q_{h,3r_h/4}),
\]
with $V\cap\exa_{j-1}\ne\varnothing$.  Thus $V\subset\exa_j$ by Step~1.

Recall that the cutoff in Step~2 satisfies $0\le\eta\le1$,
$\operatorname{supp}\eta\Subset V$, and $\|\nabla_g\eta\|_\infty\le C/r_h$. If, in addition, $\eta=1$ on $W$, then \eqref{r19}, with $L\to\infty$, gives
\begin{align}
\left(\int_Wu_j^{\beta\dcrit}\,\din v_g\right)^{2/\dcrit}
&\le C_3,\label{52}\\
\consy^{2/\dcrit}
\left(\int_{W\cap\partial\cil}u_j^{\beta\dcritb}\,\din s_g\right)^{2/\dcritb}
&\le C_3  . \label{53}
\end{align}

Retain the fixed sets $W\Subset V\subset\exa_j$ and the number $\beta>1$ from Steps~2--4.
For the exponent computations define
\[
 p_1=\frac{n\beta}{2},\qquad
 p_2       =(n-1)\beta,\qquad
 p_3  =\frac{2p_1}{p_1-1},\qquad
 p_5    =\frac{2p_2       }{p_2       -1}.
\]

Using $(\dcrit-2) p_1=\frac{4}{n-2} \frac{n \beta}{2}=\beta \dcrit$ and \eqref{52},
\begin{equation} \label{57}
\begin{aligned}
\|u_j^{\dcrit-2}\|_{L^{p_1}(W)}
= \|u_j\|_{L^{\beta\dcrit}(W)}^{\dcrit-2}\le C_3^{{    \frac{\dcrit-2}{2\beta}}}.
\end{aligned}
\end{equation}

  Since $(\dcritb-2)p_2=\frac{2}{n-2}(n-1)\beta=\beta\dcritb$, \eqref{53} gives
\begin{equation*} 
\begin{aligned}
\|[ \consy ^{1/(\beta \dcrit)}u_j]^{\dcritb-2}\|_{L^{p_2}(W\cap\partial\cil)}
=\consy ^{\frac{\dcritb-2}{\beta \dcrit}}\|u_j\|_{L^{\beta\dcritb}(W\cap\partial\cil)}^{\dcritb-2}\le C_3^{\frac{\dcritb-2}{{    2\beta}}}.
\end{aligned}
\end{equation*} 
Hence
\begin{equation} \label{59}
\|\gamma _j u_j^{\dcritb-2}\|_{L^{p_2}(W\cap\partial\cil)}
\le \mathtt{C} _5,
\end{equation} 
because $\gamma _j \le \consy ^{\dcritb/\dcrit}$, where
$$
 \mathtt{C}_5:=  \consy ^{ \frac{2}{\beta {    \dcrit}} + \frac{{\dcritb}(\beta -1)}{\dcrit \beta} }C_3^{\frac{\dcritb-2}{{    2\beta}}}.
$$
For $\consy=0$, $\gamma_j=\mathtt C_5=0$.

Let $f\in W^{1,2}(\cil)$ satisfy $\operatorname{supp}f\Subset W$. Extend $f$ by zero outside $W$.
Propositions~\ref{26}, \ref{pr} (with $q=2$), and \eqref{54} in Lemma~\ref{23} 
provide constants independent of the chosen chart such that
\[
\begin{aligned}
\|f\|_{L^{\dcrit}(\cil)}^2
&\le C_6\left(\|\nabla_g f\|_{L^2(\cil)}^2+\|f\|_{L^2(\cil)}^2\right),\\
\|f\|_{L^{\dcritb}(\partial\cil)}^2
&\le C_6\left(\|\nabla_g f\|_{L^2(\cil)}^2+\|f\|_{L^2(\cil)}^2\right),\\
\|f\|_{L^2(\partial\cil)}^2
&\le C_6\int_\cil|f||\nabla_g f|\,\din v_g
+C_6\int_\cil|f|^2\,\din v_g\\
&\le C_6\|f\|_{L^2(\cil)}\|\nabla_g f\|_{L^2(\cil)}
+C_6\|f\|_{L^2(\cil)}^2,
\end{aligned}
\]
where $C_6>0$ is independent of $f$.

Put 
\[
{\bf X}:=\|\nabla_g f\|_{L^2(\cil)}^2+\|f\|_{L^2(\cil)}^2,
\qquad {\bf Y}:=\|f\|_{L^2(\cil)}^2.
\]
Then 
\[
\begin{aligned}
\|f\|_{L^{\dcrit}(\cil)}^2
+\|f\|_{L^{\dcritb}(\partial\cil)}^2&\le 2C_6{\bf X},\\
\|f\|_{L^2(\partial\cil)}^2
&\le 2 C_6\sqrt{{\bf X}{\bf Y}}.
\end{aligned}
\]

From $\frac{p_3(1-1/\beta)}{2}+
\frac{p_3}{\beta\dcrit} =1$ and H\"older's inequality, we obtain
\[
\begin{aligned}
\int_\cil |f|^{p_3  }\,\din v_g
&=\int_\cil|f|^{p_3  (1-1/\beta)}
                   |f|^{p_3  /\beta}\,\din v_g\\
&\le\left(\int_\cil|f|^2\,\din v_g\right)^{{p_3  (1-1/\beta)}/{2}}
     \left(\int_\cil|f|^{\dcrit}\,\din v_g\right)^{{p_3  }/{(\beta\dcrit)}}.
\end{aligned}
\]
Then
\begin{equation} \label{55}
\begin{aligned}
\|f\|_{L^{p_3  }(\cil)}^2
\le {\bf Y}^{1-1/\beta}(C_6{\bf X})^{1/\beta}
\end{aligned}
\end{equation}

By $\frac{p_5    (1-1/\beta)}{2}+
\frac{p_5    }{\beta\dcritb}=1$ and H\"older's inequality, we obtain
\[
\begin{aligned}
\int_{\partial\cil}|f|^{p_5    }\,\din s_g
&=\int_{\partial\cil}|f|^{p_5    (1-1/\beta)}
                   |f|^{p_5    /\beta}\,\din s_g\\
&\le\left(\int_{\partial\cil}|f|^2\,\din s_g\right)^{{p_5    (1-1/\beta)}/{2}}
 \left(\int_{\partial\cil}|f|^{\dcritb}\,\din s_g\right)^{{p_5    }/{(\beta\dcritb)}}.
\end{aligned}
\]
Then
\begin{equation}\label{56}
\begin{aligned}
\|f\|_{L^{p_5    }(\partial\cil)}^2
&\le\|f\|_{L^2(\partial\cil)}^{2(1-1/\beta)}
\|f\|_{L^{\dcritb}(\partial\cil)}^{2/\beta}\\
&\le (2C_6{\bf X}^{1/2}{\bf Y}^{1/2})^{1-1/\beta}(C_6{\bf X})^{1/\beta}\\
&\le C_6 2^{1-1/\beta} {\bf X}^{(\beta+1)/(2\beta)}{\bf Y}^{(\beta-1)/(2\beta)}.
\end{aligned}
\end{equation}

 Since $p_1/(p_1-1)=p_3  /2$, \eqref{57} and \eqref{55} give
\begin{equation}\label{60}
\begin{aligned}
\int_Wu_j^{\dcrit-2}f^2\,\din v_g
&\le\left(\int_Wu_j^{(\dcrit-2)p_1}\,\din v_g\right)^{1/p_1}
    \left(\int_W |f|^{2p_1/(p_1-1)}\,\din v_g\right)^{(p_1-1)/p_1}\\
&=\|u_j^{\dcrit-2}\|_{L^{p_1}(W)}\|f\|_{L^{p_3  }(W)}^2\\
&\le C_3^{{    \frac{\dcrit-2}{2\beta}}}C_6^{1/\beta}{\bf X}^{1/\beta}{\bf Y}^{1-1/\beta}.
\end{aligned}
\end{equation}

Since $p_2       /(p_2       -1)=p_5    /2$, \eqref{59} and \eqref{56} give
\begin{equation}\label{61}
\begin{aligned}
\gamma _j\int_{W\cap\partial\cil}u_j^{\dcritb-2}f^2\,\din s_g
\le & \left[\int_{W\cap\partial\cil}
 \left(\gamma _j u_j^{\dcritb-2}\right)^{p_2       }\,\din s_g\right]^{1/p_2       }\\
&\cdot \left(\int_{W\cap\partial\cil}
 |f|^{2p_2       /(p_2       -1)}\,\din s_g\right)^{(p_2       -1)/p_2       }\\
=&\|\gamma _j u_j^{\dcritb-2}\|_{L^{p_2       }(W\cap\partial\cil)}
  \|f\|_{L^{p_5    }(W\cap\partial\cil)}^2\\
\le & 2^{1-1/\beta} \mathtt{C}_5C_6 {\bf X}^{(\beta+1)/(2\beta)}{\bf Y}^{(\beta-1)/(2\beta)}.
\end{aligned}
\end{equation}

By Young's inequality,
\begin{align*}
c{\bf X}^\theta {\bf Y}^{1-\theta}
&\le \frac h2{\bf X}
 +c ^{1/(1-\theta)}
 \left(\frac{2}{h}\right)^{\theta/(1-\theta)}{\bf Y}\\
&\le\frac h2{\bf X}+(2c)^{1/(1-\theta)} h^{-\theta/(1-\theta)}{\bf Y},
\end{align*}
where $0<\theta<1$, $1>h>0$, and $c\ge0$.

Applying this to \eqref{60} and \eqref{61}, we obtain
\[
\begin{aligned}
\int_Wu_j^{\dcrit-2}f^2\,\din v_g
&\le\frac h2{\bf X}+C_7 h^{-1/(\beta-1)}{\bf Y},\\
\gamma_j\int_{W\cap\partial\cil}u_j^{\dcritb-2}f^2\,\din s_g
&\le\frac h2{\bf X}+\mathtt{C} _8 h^{-(\beta+1)/(\beta-1)}{\bf Y},
\end{aligned}
\]
where $C_7>0$ and $\mathtt C_8\ge0$ are independent of $u_j,h,f$ {and of the chart, with $\beta$ fixed}. 

Therefore,
\[
\begin{aligned}
&\int_Wu_j^{\dcrit-2}f^2\,\din v_g
 +\gamma_j\int_{W\cap\partial\cil}u_j^{\dcritb-2}f^2\,\din s_g\\
&\quad\le h{\bf X}+(C_7 + \mathtt{C}_8)h^{-m}{\bf Y}\\
&\quad=h\|\nabla_gf\|_{L^2(\cil)}^2+h{\bf Y}+(C_7 + \mathtt{C}_8)h^{-m}{\bf Y}\\
&\quad\le h\|\nabla_gf\|_{L^2(\cil)}^2+(C_7 + \mathtt{C}_8+1)(1+h^{-m})\|f\|_{L^2(\cil)}^2,
\end{aligned}
\]
where $m=(\beta+1)/(\beta-1)>1$. This proves \eqref{r22}.


\medskip

\noindent\textbf{Step 5.} Let $k\ge1$, $L>0$, and choose a smooth cutoff $\zeta$ with
$0\le\zeta\le1$, $\zeta =1$ on $W^\prime \Subset W$, and $\operatorname{supp}\zeta\Subset W$.

 We will show that 
\begin{equation}\label{r23}
\begin{aligned}
\|u_j\|_{L^{\dcrit k}(W')}^{2k}
&=\left(\int_{W'}u_j^{\dcrit k}\,\din v_g\right)^{2/\dcrit}\\
&\le C_9 k^{m_1}(1+\|\nabla_g\zeta\|_{L^\infty(\cil)}^2)
\|u_j\|_{L^{2k}(\operatorname{supp}\zeta)}^{2k},
\qquad m_1=\frac{4\beta}{\beta-1},
\end{aligned}
\end{equation}
where $C_9>0$ is independent of $j,k,V$.

Put
\[
 u_{j,L}=\min\{u_j,L\},\qquad
 w_L=u_ju_{j,L}^{k-1},\qquad 
 \varphi_L=\zeta^2u_ju_{j,L}^{2k-2}.
\]

By \eqref{41}, with $\beta=k$ and $\delta=\theta_k$,
\[
\begin{aligned}
\frac3{5k}\int_W|\nabla_g (\zeta w_L)|^2\,\din v_g
 &\le\int_W\langle\nabla_g u_j,\nabla_g\varphi_L\rangle_g\,\din v_g+7k\int_Ww_L^2|\nabla_g\zeta|^2\,\din v_g.
\end{aligned}
\]
since $k\theta_k\ge1$,
\[
\begin{aligned}
\frac{\theta_k-\delta/4}{1+\delta/4} =\frac{3\theta_k/4}{1+\theta_k/4} \ge \frac{3\theta_k}{5}\ge\frac3{5k},
\end{aligned}
\]
and
\[
\begin{aligned}
\frac4{\delta}\left(1+\theta_k-\frac{\delta}{4}\right)
 =\frac4{\theta_k}+3\le4k+3\le7k.
\end{aligned}
\]

By \eqref{40},
\[
\begin{aligned}
\frac3{5k}\|\nabla_g (\zeta w_L)\|_{L^2(W)}^2
 &\le{\lambdaj ^+}   \left[
 \int_Wu_j^{\dcrit-2}(\zeta w_L)^2\,\din v_g
 +\gamma_j\int_{W\cap\partial\cil}u_j^{\dcritb-2}(\zeta w_L)^2\,\din s_g
 \right]\\
 &\quad+C_1\|\zeta w_L\|_{L^2(W)}^2
 +C_2\|\zeta w_L\|_{L^2(W\cap \partial\cil)}^2
 +7k\int_Ww_L^2|\nabla_g\zeta|^2\,\din v_g.
\end{aligned}
\]
Here $C_1,C_2$ are precisely the curvature constants in
\eqref{40}. 

Then, since $\sup _{j\ge j_0}\lambdaj^+ <\infty$ by \eqref{r3},
\begin{equation}\label{64}
\begin{aligned}
\|\nabla_g (\zeta w_L)\|_{L^2(W)} ^2
 \le &C_{10}k^2\Biggl[
 \int_Wu_j^{\dcrit-2}(\zeta w_L)^2\,\din v_g
 +\gamma_j\int_{W\cap\partial\cil}u_j^{\dcritb-2}(\zeta w_L)^2\,\din s_g\\
 &+\|\zeta w_L\|_{L^2(W)}^2
 +\|\zeta w_L\|_{L^2(W\cap \partial\cil)}^2
 +\int_Ww_L^2|\nabla_g\zeta|^2\,\din v_g\Biggr].
\end{aligned}
\end{equation}
The constant $C_{10}$ in this inequality is independent of
$j,h,i,k,L$.

Applying \eqref{r22} and \eqref{n3} to $\zeta w_L$, we obtain, for $0<h<1$ and $\tau>0$,
\begin{equation}\label{63}
\begin{aligned}
\int_Wu_j^{\dcrit-2}(\zeta w_L)^2\,\din v_g
+\gamma_j\int_{W\cap\partial\cil}u_j^{\dcritb-2}(\zeta w_L)^2\,\din s_g
\le & h\|\nabla_g(\zeta w_L)\|_{L^2(\cil)}^2+C_4(1+h^{-m})\|\zeta w_L\|_{L^2(\cil)}^2,\\
\|\zeta w_L\|_{L^2(\partial\cil)}^2
\le & \tau\|\nabla_g(\zeta w_L)\|_{L^2(\cil)}^2
+C_{11}(1+\tau^{-1})\|\zeta w_L\|_{L^2(\cil)}^2,  
\end{aligned}
\end{equation}
where the constant $C_{11}$ is independent of
$j,h,i,k,L$.

From \eqref{64} and \eqref{63},
\[
\begin{aligned}
\|\nabla_g(\zeta w_L)\|_{L^2(\cil)} ^2
&\le C_{12}k^2(h+\tau)\|\nabla_g(\zeta w_L)\|_{L^2(\cil)} ^2\\
&\quad+C_{12}k^2(1+h^{-m}+\tau^{-1})\|\zeta w_L\|_{L^2(\cil)}  ^2
 +C_{12}k^2\int_Ww_L^2|\nabla_g\zeta|^2\,\din v_g,
\end{aligned}
\]
where $C_{12}:=C_{10}[1+2(C_4+C_{11})]$.

For $k\ge1$, choose 
\[
 h=\tau=\frac1{4(C_{12} +1)k^2} 
\]
Then
\[
\begin{aligned}
 k^2h^{-m}&
=[4(C_{12}+1)]^m k^{2+2m},\\
 k^2\tau^{-1 }&\le4(C_{12}+1)k^{2+2m},\\
\end{aligned}
\]
where $m=\frac{\beta+1}{\beta-1}$ and $\beta$ is the fixed number from Step~3.

Then
\[
\begin{aligned}
\|\nabla_g(\zeta w_L)\|_{L^2(\cil)}^2
&\le C_{13}k^{m_1}\int_W
 (\zeta^2+|\nabla_g\zeta|^2)w_L^2\,\din v_g,
\end{aligned}
\]
where $C_{13}=6{C_{12}}[4(C_{12}+1)]^m$ and
$$m_1:=2+2m
=\frac{4\beta}{\beta-1}.$$

Applying \eqref{r15} to $\zeta w_L$ gives
\[
\begin{aligned}
(\Lambda^+_{n,\consy}-\epsilon)
\|\zeta w_L\|_{L^{\dcrit}(V)}^2
&\le(\Lambda^+_{n,\consy}-\epsilon)
\left(\|\zeta w_L\|_{L^{\dcrit}(V)}^{\dcrit}
+\consy\|\zeta w_L\|_{L^{\dcritb}(V\cap\partial\cil)}^{\dcrit}
\right)^{2/\dcrit}\\
&\le\|\nabla_g(\zeta w_L)\|_{L^2(V)}^2\\
&\le C_{13}k^{m_1}(1+\|\nabla_g\zeta\|_{L^\infty(\cil)}^2)
\int_{\operatorname{supp}\zeta}w_L^2\,\din v_g.
\end{aligned}
\]
If $\zeta=1$ on $W'\Subset W$, this yields
\[
\left(\int_{W'}w_L^{\dcrit}\,\din v_g\right)^{2/\dcrit}
\le {\frac{C_{13}}{\Lambda^+_{n,\consy}-\epsilon}}k^{m_1}(1+\|\nabla_g\zeta\|_{L^\infty(\cil)}^2)
\int_{\operatorname{supp}\zeta}w_L^2\,\din v_g.
\]

Starting with $k=1$ and letting $L\to\infty$, we obtain
\begin{equation*}
\begin{aligned}
\|u_j\|_{L^{\dcrit k}(W')}^{2k}
&=\left(\int_{W'}u_j^{\dcrit k}\,\din v_g\right)^{2/\dcrit}\\
&\le {\frac{C_{13}}{\Lambda^+_{n,\consy}-\epsilon}}k^{m_1}(1+\|\nabla_g\zeta\|_{L^\infty(\cil)}^2)
\|u_j\|_{L^{2k}(\operatorname{supp}\zeta)}^{2k},
\end{aligned}
\end{equation*}
with $C_9=C_{13}/(\Lambda^+_{n,\consy}-\epsilon)$. Proceeding inductively, we obtain \eqref{r23}.

\medskip

\noindent\textbf{Step 6.}
Fix $V,W$ as in Step~4 and put
\[
\rho_\ell=\frac12+\frac1{4\cdot2^\ell}, \quad
 D_\ell:=\Phi_{h,z_{h,i}}(Q_{h,\rho_\ell r_h}),
\quad\ell=0,1,2,\ldots.
\]

 Direct subtraction gives
\begin{equation}\label{78}
\begin{aligned}
\rho_\ell-\rho_{\ell+1}
&=\frac1{2^{\ell+2}}-\frac1{2^{\ell+3}}
=\frac1{2^{\ell+3}}.
\end{aligned}
\end{equation}

Let $\psi$ be the smooth cutoff denoted by $\beta$ in Step~4 of Lemma~\ref{pt}. Thus
\[
 0\le\psi\le1,\qquad \psi(t)=1\ (t\le1/2),\qquad
 \psi(t)=0\ (t\ge3/4),\qquad \|\psi'\|_{L^\infty({    \mathbb R})}<\infty.
\]

Define
\[
 \psi_\ell(t):=\psi\left(\frac12+\frac{t-\rho_{\ell+1}r_h}{2(\rho_\ell-\rho_{\ell+1})r_h}\right).
\]

Then
\begin{gather*}
 \psi_\ell(t)=1 \ \text { if }t\le\rho_{\ell+1}r_h, \qquad
 \psi_\ell(t)=0 \ \text { if }  t\ge\frac{\rho_\ell+\rho_{\ell+1}}2r_h,\\
 |\psi_\ell'(t)|\le\frac{\|\psi'\|_{L^\infty({    \mathbb R})}}{2(\rho_\ell-\rho_{\ell+1})r_h}.
\end{gather*}

In the chosen chart, set
\[
\begin{aligned}
 \eta_\ell(\Phi_{1,z_{1,i}}(\xeuc',t))&=\psi_\ell(|\xeuc'|)\psi_\ell(t) &&(h=1),\\
 \eta_\ell(\Phi_{0,z_{0,i}}(\xeuc))&=\psi_\ell(|\xeuc|) &&(h=0),
\end{aligned}
\]
and extend by zero outside the chart. 

We have
\[
 \spt \eta_\ell
 \subset\Phi_{h,z_{h,i}}\left(\overline Q_{h,\frac12(\rho_\ell+\rho_{\ell+1})r_h}\right)
 \Subset\Phi_{h,z_{h,i}}(Q_{h,\rho_\ell r_h})=D_\ell \subset W
\]
and
\[
\begin{aligned}
 |\nabla\eta_\ell|^2&\le\frac{2\|\psi'\|_{L^\infty({    \mathbb R})}  ^2}{4(\rho_\ell-\rho_{\ell+1})^2r_h^2},\\
 |\nabla_g\eta_\ell|^2&\le\radpa|\nabla\eta_\ell|^2
 \le\frac{\radpa\|\psi'\| _{L^\infty({    \mathbb R})} ^2}{2(\rho_\ell-\rho_{\ell+1})^2r_h^2},
\end{aligned}
\]
see \eqref{bgr}.

Then, by \eqref{78},
\[
\begin{aligned}
\|\nabla_g\eta_\ell\|_{L^\infty(\cil)}
\le C_{14} 2^{\ell +3},
\end{aligned}
\]
where the last constant $C_{14}>1$ is independent of $\ell$, $V$, and $u_j$.

Set
\[
\mathtt{A}:=\frac{\dcrit}{2}=\frac{n}{n-2}>1,
\qquad k_\ell:=\mathtt{A}^\ell.
\]

Then
\[
\dcrit k_\ell=2\mathtt{A} k_\ell=2k_{\ell+1},
\qquad k_\ell\ge1.
\]
Apply \eqref{r23} with
$k=k_\ell$, $W'=D_{\ell+1}$, and ${\zeta}=\eta_\ell$,
\[
\begin{aligned}
\|u_j\|_{L^{2k_{\ell+1}}(D_{\ell+1})}^{2k_\ell}
&=\|u_j\|_{L^{\dcrit k_\ell}(D_{\ell+1})}^{2k_\ell}\\
&\le C _9k_\ell^{m_1}
 (1+\|\nabla_g\eta_\ell\|_{L^\infty(\cil)} ^2)
 \|u_j\|_{L^{2k_\ell}(D_\ell)}^{2k_\ell}\\
&\le C_{15} 4^\ell k_\ell^{m_1}
 \|u_j\|_{L^{2k_\ell}(D_\ell)}^{2k_\ell},
\end{aligned}
\]
where $C_{15}:=(C_9+1) C_{14}^{{    2}} 2^7>1$.

Hence
\[
\|u_j\|_{L^{2k_{\ell+1}}(D_{\ell+1})}
\le(C_{15}4^\ell k_\ell^{m_1})^{1/(2k_\ell)}
\|u_j\|_{L^{2k_\ell}(D_\ell)}.
\]
Consequently, by induction, 
\begin{equation} \label{79}
\begin{aligned}
\|u_j\|_{L^{2k_{\overline{m}}}(D_{\overline{m}})}
&\le \|u_j\|_{L^2(D_0)}
 \prod_{\ell=0}^{\overline{m}-1}
 (C_{15}4^\ell k_\ell^{m_1})^{1/(2k_\ell)}\\
&\le \vol _g (V) ^{1/n}\|u_j\|_{L^{\dcrit} (\exa _j)}\prod_{\ell=0}^{\overline{m}-1}
 (C_{15}4^\ell\mathtt{A}^{m_1\ell})^{1/(2k_\ell)}\\
&\le C_{16}\prod_{\ell=0}^{\overline{m}-1}
 (C_{15}4^\ell\mathtt{A}^{m_1\ell})^{1/(2k_\ell)},
\end{aligned}
\end{equation}
where $C_{16}:= \vol _g (V)^{1/n}$, see \eqref{r10}. {    By \eqref{51},
\[
C_{16}\le\radpa^{1/2}\max_{h=0,1}\operatorname{Vol}(Q_{h,r_h})^{1/n},
\]
uniformly in the chart.}

We also have
\[
\begin{aligned}
&\log\prod_{\ell=0}^{\overline{m}-1}
 (C_{15}4^\ell k_\ell^{m_1})^{1/(2k_\ell)}\\
&=\frac12\sum_{\ell=0}^{\overline{m}-1}
 \frac{\log C_{15}+\ell\log4+m_1\log k_\ell}{k_\ell}\\
&=\frac12\sum_{\ell=0}^{\overline{m}-1}
 \frac{\log C_{15}+\ell\log4+m_1\ell\log\mathtt{A}}{\mathtt{A}^\ell}\\
&\le\frac{\log C_{{    15}}}{2}\sum_{\ell=0}^\infty\mathtt{A}^{-\ell}
 +\frac{\log4+m_1\log\mathtt{A}}{2}
 \sum_{\ell=0}^\infty\ell\mathtt{A}^{-\ell}\\
&=\frac{\log C_{15}}{2}\frac{\mathtt{A}}{\mathtt{A}-1}
 +\frac{\log4+m_1\log\mathtt{A}}{2}
 \frac{\mathtt{A}}{(\mathtt{A}-1)^2}=: C_{17}.
\end{aligned}
\]

Letting $\overline{m}\to \infty$ in \eqref{79}, we obtain
\[
\begin{aligned}
\|u_j\|_{L^\infty(\Phi_{h,z_{h,i}}(Q_{h,r_h/2}))}
&=\lim_{m\to\infty}
\|u_j\|_{L^{2k_m}(\Phi_{h,z_{h,i}}(Q_{h,r_h/2}))}\le C_{16}e^{C_{17}}.
\end{aligned}
\]
Together with \eqref{50}, this completes the proof of the proposition.
\end{proof}

\section{Proof of Theorem \ref{r25}}

The global construction below uses an exhaustion and adapts the nonvanishing argument of Wei \cite[Proposition~3.2]{Wei2019}; see also \cite[Lemma~12]{Grosse2013}. A diagonal argument yields a local limit of the exhaustion solutions. If this limit vanished, the normalized mass would escape to infinity, contradicting the strict gap between the global variational level and the level at infinity.
\begin{proof}
\textbf{Step 1.} By Proposition~\ref{r12}, there exist constants $j_0>0$ and $C>0$ such that
\begin{equation}\label{80}
\sup _{j\ge j_0} \|u_j\|_{L^\infty (\exa _{j-1})} <C.
\end{equation}

Moreover, by \eqref{r5},
\begin{equation}\label{81}
\|u_j\|_{L^{\dcrit} (\exa _j)} ^{\dcrit} + \consy  \|u_j\|_{L^{\dcritb} (\exa _j\cap \partial \cil)} ^{\dcrit} =1 .
\end{equation}

Let $L\in \mathbb{N}$ satisfy $L>j_0$. Choose
$\xi\in C_c^\infty(\exa _{L+1})$ such that $\xi|_{\exa _{L}}=1$ and $0\le \xi \le 1$. For $j>L+1$, we have
\[
\begin{aligned}
\langle\nabla_gu_j,\nabla_g(\xi^2u_j)\rangle_g
&=\xi^2|\nabla_gu_j|^2+2\xi u_j\langle\nabla_gu_j,\nabla_g\xi\rangle_g.
\end{aligned}
\]

For the second term,
\[
\begin{aligned}
2\left|\int_{\cil}\xi u_j\langle\nabla_gu_j,\nabla_g\xi\rangle_g\,\din v_g\right|
&\le2\left(\int_{\cil}\xi^2|\nabla_gu_j|^2\,\din v_g\right)^{1/2}
\left(\int_{\cil}u_j^2|\nabla_g\xi|^2\,\din v_g\right)^{1/2}\\
&\le\frac12\int_{\cil}\xi^2|\nabla_gu_j|^2\,\din v_g
+2\int_{\cil}u_j^2|\nabla_g\xi|^2\,\din v_g.
\end{aligned}
\]

Using $\xi ^2 u_j$ as a test function in \eqref{24}, we obtain
\[
\begin{aligned}
 \int_{\exa _j} \xi^2|\nabla_gu_j|^2\,\din v_g
 &=\lambdaj\left[\int_{\exa _j}\xi^2u_j^{\dcrit}\,\din v_g
 +\gamma_j\int_{{\exa _j}\cap \partial\cil}\xi^2u_j^{\dcritb}\,\din s_g\right]\\
 &\quad-2\int_{\exa _j} \xi u_j\langle\nabla_gu_j,\nabla_g\xi\rangle_g\,\din v_g
 -\frac{n-2}{4(n-1)}\int_{\exa _j} R_g\xi^2u_j^2\,\din v_g\\
 &\quad-\frac{n-2}{2(n-1)}\int_{{\exa _j} \cap \partial\cil}H_g\xi^2u_j^2\,\din s_g.
\end{aligned}
\]

The last two estimates, together with \eqref{80}, imply that, for $j>L$,
\begin{equation*}
\int_{B_L} |\nabla_gu_j|^2\,\din v_g \le C_1,
\end{equation*}
where $C_1>0$ is independent of $j$.

Hence there exists a subsequence $\{u_{L_j}\}$ such that
\begin{align*}
&u_{L_j} \rightharpoonup u_L   \qquad \text {in }{    W^{1,2}}(\exa _L)\\
&u_{L_j} \to u_L \qquad   \text { in } L^{2} (\exa _L) \text { and } a.e.,\\
&u_{L_j} \to u_L \qquad  \text { in } L^{2} (\exa _L \cap \partial \cil) \text { and } a.e.
\end{align*}

By a diagonal argument, there exist a subsequence $\{u_m\}$ and a function $u\in W^{1,2}_{\loc} (\cil)$ such that
\begin{align*}
&u_{m} \rightharpoonup u \qquad \text {in }W^{1,2} _{\loc} (\cil)\\
&u_{m} \to u \qquad   \text { in } L^{2}_{\loc} (\cil),\\
&u_{m} \to u \qquad  \text { in } L^{2} _{\loc}( \partial \cil)\\
&u_{m} \to u \qquad  a.e \text { in } \cil .
\end{align*}

Furthermore, by \eqref{80} and \eqref{81},
$$
u\in L^{\infty} (\cil) \cap L^{\dcrit} (\cil).
$$

If $\consy >0$, we may also use \eqref{81} to assume that
$$
\|u_m\|_{L^{\dcritb} ({    \partial\cil})} \to \overline{d} \in [0,\consy ^{-1/\dcrit}].
$$

Therefore, $u$ is a weak solution of \eqref{r26}.

\medskip 

\noindent {\bf Step 2.} We now show that $u\not\equiv 0$. Suppose, for a contradiction, that the limit vanishes identically.

Choose $\psi\in C_c^\infty(\exa_{L+1})$ with $0\le\psi\le\pi/2$ and $\psi=\pi/2$ on a neighborhood of $\overline{\exa_L}$. Set
\[
 \chi_0=\sin\psi,\qquad \chi_1=\cos\psi.
\]
Then $\chi_0=1$ and $\chi_1=0$ near $\overline{\exa_L}$, while $\chi_0=0$ and $\chi_1=1$ outside $\operatorname{supp}\psi$.

We have
\begin{gather*}
 \chi_0^2+\chi_1^2=1,\\
 \chi_0\nabla_g\chi_0+\chi_1\nabla_g\chi_1=0,\qquad
 |\nabla_g\chi_0|^2+|\nabla_g\chi_1|^2=|\nabla_g\psi|^2.
\end{gather*}
Then, for $j>L$,
\[
\begin{aligned}
 \sum_{a=0}^1|\nabla_g(\chi_a u_j)|^2
 &=\sum_{a=0}^1\left(\chi_a^2|\nabla_gu_j|^2
 +u_j^2|\nabla_g\chi_a|^2
 +2\chi_a u_j\langle\nabla_gu_j,\nabla_g\chi_a\rangle_g\right)\\
 &=|\nabla_gu_j|^2+u_j^2|\nabla_g\psi|^2.
\end{aligned}
\]

Hence
\[
 E_g(\chi_0u_j)+E_g(\chi_1u_j)
 =E_g(u_j)+\int_{\operatorname{supp}\psi}u_j^2|\nabla_g\psi|^2\,\din v_g.
\]
We also have
\begin{equation}\label{85}
\begin{aligned}
E_g(\chi_1u_j)
=& \lambdaj+\int_{\operatorname{supp}\psi}u_j^2|\nabla_g\psi|^2\,\din v_g-E_g(\chi_0u_j)\\
\le & \lambdaj+
\left(\|\nabla_g\psi\|_{L^\infty (\cil)} ^2+\frac{n-2}{4(n-1)}\|R_g^-\|_{L^\infty (\cil)}\right)
\|u_j\|_{L^2(\operatorname{supp}\psi)}^2\\
& + \frac{n-2}{2(n-1)}\|H_g^-\|_{L^\infty ({    \partial}\cil)}
\|u_j\|_{L^2(\operatorname{supp}\psi\cap\partial\cil)}^2\\
\le & \lambdaj + o({    1}).
\end{aligned}
\end{equation}

We now prove that
\begin{equation}\label{82}
1-\opnj (\chi _1 u_j) = \opnj ( u_j) -\opnj (\chi _1 u_j) \to 0.
\end{equation}

First, observe that
\begin{equation}\label{83}
\begin{aligned}
 0\le&\int_\cil u_j^{\dcrit}\,\din v_g
 -\int_\cil(\chi_1u_j)^{\dcrit}\,\din v_g\\
 =&\int_{\operatorname{supp}\psi}(1-\chi_1^{\dcrit})u_j^{\dcrit}\,\din v_g\\
 \le &\int_{\operatorname{supp}\psi}u_j^{\dcrit}\,\din v_g{\to}  0.
\end{aligned}
\end{equation}

On the other hand, if $0\le a\le b$, the mean value theorem gives
$$
b^{\dcrit/\dcritb} -a^{\dcrit/\dcritb} \le \frac{\dcrit}{\dcritb}  b^{\dcrit/\dcritb-1} (b-a). 
$$
Hence, by \eqref{81}, for $\consy>0$,
\begin{equation}\label{84}
\begin{aligned}
& \consy\left(\|u_j\|_{L^{\dcritb}(\exa _j \cap {    \partial}\cil) }^{\dcrit} - \|\chi _1u_j\|_{L^{\dcritb}(\exa _j \cap {    \partial}\cil) }^{\dcrit} \right)\\
&\le  \frac{\dcrit}{\dcritb}  \consy \left( \int _{\exa _j \cap {    \partial}\cil}  u_j^{\dcritb} \, \din s _g\right)^{\dcrit/\dcritb-1}  \left[ \int _{\exa _j \cap {    \partial}\cil}  u_j ^{\dcritb}\, \din s _g -  \int _{\exa _j \cap {    \partial}\cil} (\chi _1 u_j)^{\dcritb} \, \din s _g\right] \\
&\le  \frac{\dcrit}{\dcritb}  \consy^{\dcritb/\dcrit}  \int _{\spt \psi \cap {    \partial}\cil}  (1- \chi _1^{\dcritb} )u_j^{\dcritb} \, \din s _g  \to 0.
\end{aligned}
\end{equation}

This proves \eqref{82}.

We can now complete the contradiction argument.

By \eqref{85}, \eqref{82}, Lemma~\ref{l5}, and
\[
\begin{aligned}
Y_{\consy}(\cil\setminus \overline{\exa_L})
\le& \frac{E_g(\chi_1u_j)}{{\opnj}(\chi_1u_j)^{2/\dcrit}}\\
\le & \frac{\lambdaj+o(1)}{{\opnj}(\chi_1u_j)^{2/\dcrit}},
\end{aligned}
\]
we obtain $Y_{\consy}(\cil\setminus \overline{\exa_L}) \le Y_{\consy}(\cil)$. Letting $L\to \infty$, we obtain
\[
Y_{\consy,\infty}(\cil)\le Y_{\consy}(\cil).
\]
This contradicts the strict inequality $Y_{\consy}(\cil)<{    Y_{\consy,\infty}(\cil)}$. We therefore conclude that $u\not\equiv 0$.

The right-hand sides in \eqref{r26} are locally bounded, and the metric and boundary are smooth. By \cite{Nittka2011}, $u\in C^{0,\alpha}_{\loc}(\cil)$ for some local exponent $0<\alpha<1$. The Robin data are then locally H\"older continuous. Moreover, \cite{Lieberman1988} yields $u\in C^{1,\alpha'}_{\loc}(\cil)$. Since $u\ge0$,
\[
\begin{aligned}
&\Delta_g u-\left(\frac{n-2}{4(n-1)}R_g-Y_{\consy}(\cil)u^{\dcrit-2}\right)^+u\\
&\qquad=-\left(\frac{n-2}{4(n-1)}R_g-Y_{\consy}(\cil)u^{\dcrit-2}\right)^-u\le0.
\end{aligned}
\]
The coefficient is locally bounded. Since {    $u\not\equiv0$}, the strong minimum principle and the version of Hopf's lemma in \cite{Rosales2019}, applied locally to functions in $C^1({    B_1(0)})\cap C^0({    \overline{B_1(0)}}){    \cap W^{1,2}(B_1(0))}$, give $u>0$ on $\cil$. Finally, applying Cherrier's regularity results \cite{zbMATH03879780}, we obtain $u\in C^\infty(\cil)$. \end{proof}

%


\vspace{1cm}


\noindent {\bf Funding:} This work was supported by the Conselho Nacional de Desenvolvimento Científico e Tecnológico (CNPq), Grant No. 150680/2025-2.

\noindent {\bf Data Availability:} No data were used for the research described in the article.

\noindent {\bf Declarations}

\noindent {\bf Conflict of interest:} The author declares no conflict of interest.


\noindent {\bf Declaration of generative AI and AI-assisted technologies in the writing process:} During the preparation of this work, the author(s) used ChatGPT-6.1 Sol to assist with calculations, the review of mathematical arguments and calculations, and language editing, including grammar and spelling corrections. After using this tool, the author(s) reviewed and edited the content as needed and take full responsibility for the content of the published article.


 
 \bibliographystyle{abbrv}

   \bibliography{ref}
\end{document}